\documentclass[12pt, reqno, a4paper]{amsart}
\usepackage{enumerate, amssymb}
\usepackage{amscd}
\usepackage{eufrak}

\newtheorem{theorem}{Theorem}[section]
\newtheorem{proposition}[theorem]{Proposition}
\newtheorem{lemma}[theorem]{Lemma}
\theoremstyle{definition}
\newtheorem{definition}[theorem]{Definition}

\newtheorem{corollary}[theorem]{Corollary}
\newtheorem{remark}[theorem]{Remark}

\numberwithin{equation}{section}
\def\sqr#1#2{{\,\vcenter{\vbox{\hrule height.#2pt\hbox{\vrule width.#2pt
height#1pt \kern#1pt\vrule width.#2pt}\hrule height.#2pt}}\,}}

\begin{document}
\title[Positively $p$-nuclear operators, positively $p$-integral operators and approximation properties]{Positively $p$-nuclear operators, positively $p$-integral operators and approximation properties}
\author{Dongyang Chen*}
\address{School of Mathematical Sciences, Xiamen University, Xiamen, 361005, China}
\email{cdy@xmu.edu.cn}
\author{Amar Belacel}
\address{Laboratory of Pure and Applied Mathematics (LMPA), University of Laghouat, Laghouat, Algeria}
\email{a.belacel@lagh-univ.dz}
\author{Javier Alejandro Ch\'{a}vez-Dom\'{i}nguez}
\address{Department of Mathematics, University of Oklahoma, Norman, Oklahoma, 73019, USA}
\email{jachavezd@math.ou.edu}
\thanks{*Corresponding author \\
Dongyang Chen was supported by the National Natural Science Foundation of China (Grant No. 11971403) and the Natural Science Foundation of Fujian Province of China (Grant No. 2019J01024).}

\begin{abstract}
In the present paper, we introduce and investigate a new class of positively $p$-nuclear operators that are positive analogues of right $p$-nuclear operators.
One of our main results establishes an identification of the dual space of positively $p$-nuclear operators with the class of positive $p$-majorizing operators that is a dual notion of positive $p$-summing operators.
As applications, we prove the duality relationships between latticially $p$-nuclear operators introduced by O. I. Zhukova and positively $p$-nuclear operators. We also introduce a new concept of positively $p$-integral operators via positively $p$-nuclear operators and prove that the inclusion map from $L_{p^{*}}(\mu)$ to $L_{1}(\mu)$($\mu$ finite) is positively $p$-integral.
New characterizations of latticially $p$-integral operators by O. I. Zhukova and positively $p$-integral operators are presented and used to prove that an operator is latticially $p$-integral (resp. positively $p$-integral) precisely when its second adjoint is.
Finally, we describe the space of positively $p^{*}$-integral operators as the dual of the $\|\cdot\|_{\Upsilon_{p}}$-closure of the subspace of finite rank operators in the space of positive $p$-majorizing operators. Approximation properties, even positive approximation properties, are needed in establishing main identifications.
\end{abstract}

\maketitle

\section{Introduction}

Introduced first by A. Grothendieck in \cite{G2}, the theory of $p$-summing operators was exhaustively studied by A. Pietsch \cite{Pi} and J. Lindenstrauss and A. Pe{\l}czy\'{n}ski \cite{LP}. In 1955, A. Grothendieck \cite{Gro} introduced and studied nuclear and integral operators that are central to his theory of tensor products. A. Persson and A. Pietsch \cite{PP} introduced and investigated $p$-nuclear and $p$-integral operators that are natural generalizations to arbitrary $1\leq p\leq \infty$ of the classes of nuclear operators and integral operators. The classes of $p$-summing, $p$-nuclear and $p$-integral operators have extreme utility in the study of many different problems in Banach space theory. We recommend \cite{DJT} and \cite{P} for a complete study of the topics. So it is natural to generalize these three classes of operator to various settings. In 1998, the generalization of the theory of $p$-summing operators to the noncommutative setting was first developed by G. Pisier \cite{Pis} by means of the so called \textit{completely $p$-summing maps}. Successively, the classes of nuclear operators, integral operators and other ideals of operators were generalized to the noncommutative setting ([12,19] etc.). In 2009, J. Farmer and W. B. Johnson started in \cite{FJ09} studying the $p$-summing operators in the nonlinear setting, which they called \textit{Lipschitz $p$-summing operators}. The paper \cite{FJ09} has motivated the study of various classes of classical operator ideals in the nonlinear setting (see, for instance, \cite{JMS},
\cite{Dom11}, \cite{CZ12}, \cite{BC}, etc).

By comparison to the noncommutative setting and the nonlinear setting, it seems that the theory of $p$-summing, $p$-nuclear and $p$-integral operators in the Banach lattice setting attracts much less attention.
In 1971, U. Schlotterbeck \cite{Schl} (see also \cite{Sch}) characterized abstract $M$-spaces ($AM$-spaces for short) and abstract $L$-spaces ($AL$-spaces) in a way quite different from the classical Kakutani's representation theorems for $AM$-spaces with a unit and $AL$-spaces: A Banach lattice $X$ is isometric lattice isomorphic to an $AL$-space ($AM$-space, respectively) if and only if every positive unconditionally summable sequence in $X$ is absolutely summable (every norm null sequence in $X$ is order bounded), that is, the identity map $I_{X}$ on $X$ takes positive unconditionally summable sequences to absolutely summable sequences ($I_{X}$ takes norm null sequences to order bounded sequences). In 1972, H. H. Schaefer \cite{Sch1} generalized this property of the identity map on $AL$-spaces ($AM$-spaces, respectively) in a natural way and introduced the concept of the so called \textit{cone absolutely summing operators} (\textit{majorizing operators}, respectively). Furthermore, H. H. Schaefer \cite{Sch1} characterized cone absolutely summing operators (majorizing operators, respectively) by factoring positively through $AL$-spaces ($AM$-spaces, respectively). On the other hand, by introducing the $l$-norm on the class of all cone absolutely summing operators (the $m$-norm on the class of all majorizing operators), H. H. Schaefer \cite{Sch1} extended Schlotterbeck's characterizations of $AL$-spaces ($AM$-spaces, respectively).
In 1971, L. Krsteva in \cite{Kr} written in Russian (see also \cite{GC}) extended cone absolutely summing operators to the so-called \textit{latticially $p$-summing operators}.
Being unaware of \cite{Kr} and \cite{GC}, O. Blasco ([3,2]) introduced the notion of positive $p$-summing operators, which is exactly the same as latticially $p$-summing operators.
Having latticially $p$-summing operators at hand, it is natural to think about $p$-nuclear and $p$-integral operators in the Banach lattice setting.
In 1998, O. I. Zhukova \cite{Zhu} defined and investigated a partially positive version of $p$-nuclear operators-\textit{latticially $p$-nuclear operators}. By using of latticially $p$-nuclear operators, O. I. Zhukova \cite{Zhu} naturally introduced the notion of latticially $p$-integral operators and proved
some of well-known results analogous to the classical theory of $p$-summing, $p$-nuclear and $p$-integral operators.

This paper is a continuous work of \cite{CBC}.
The aim of the present paper is to develop the theory of $p$-nuclear and $p$-integral operators in the Banach lattice setting. The paper is organized as follows.

It was known \cite[Theorem 18.2.5]{P} that the adjoint operator ideal $[\mathcal{N}_{p},\nu_{p}]^{*}$ of $[\mathcal{N}_{p},\nu_{p}]$ is equal to $[\prod_{p^{*}},\pi_{p^{*}}]$. This formula described the dual space $(\mathcal{N}_{p}(E,F))^{*}$ as the space $\prod_{p^{*}}(F,E^{**})$ if $E^{*}$ and $F$ have the metric approximation property. O. I. Zhukova \cite{Zhu} established an analogous representation theorem for $(\widetilde{\mathcal{N}}_{p}(E,X))^{*}$, the dual space of the latticially $p$-nuclear operators, in terms of latticially $p$-summing operators if $E^{*}$ has the metric approximation property or $X$ has the positive metric approximation property. In Section 2, we introduce the notion of positively $p$-nuclear operators that is a partially positive version of right $p$-nuclear operators (\cite{Per},\cite[Sec.6.2]{R}). Firstly, we show that the class of positively $p$-nuclear operators does not coincide with the class of right $p$-nuclear operators. Secondly, we establish a representation theorem for $(\widetilde{\mathcal{N}}^{p}(X,E))^{*}$, the dual space of the positively $p$-nuclear operators, by means of positive $p$-majorizing operators introduced by D. Chen, A. Belacel and J. A. Ch\'{a}vez-Dom\'{i}nguez \cite{CBC} if $E$ has the approximation property or $X^{*}$ has the positive metric approximation property. Recall that when $E^{*}$ has the approximation property, any operator $T:E\rightarrow F$ with nuclear adjoint is nuclear and both nuclear norms coincide (see for instance \cite[Proposition 4.10]{R}). The analogous result for $p$-nuclear operators due to O. I. Reinov \cite[Theorem 1]{Rei} states that when $E^{*}$ or $F^{***}$ has the approximation property, then an operator $T:E \rightarrow F$ with $p$-nuclear adjoint is right $p$-nuclear and the $p$-nuclear norm of $T^{*}$ and the right $p$-nuclear norm of $T$ coincide.
As a corollary of our representation theorem for $(\widetilde{\mathcal{N}}^{p}(X,E))^{*}$, we prove that when $E^{***}$ has the approximation property or $X^{*}$ has the positive metric approximation property, then an operator $T:X\rightarrow E$ with a latticially $p$-nuclear adjoint is positively $p$-nuclear and the latticially $p$-nuclear norm of $T^{*}$ and the positively $p$-nuclear norm of $T$ coincide. Furthermore, we use O. I. Zhukova's representation theorem for $(\widetilde{\mathcal{N}}_{p}(E,X))^{*}$ to prove that when $E^{*}$ has the approximation property or $X^{****}$ has the positive metric approximation property, then an operator $S:E\rightarrow X$ with a positively $p$-nuclear adjoint is latticially $p$-nuclear and the positively $p$-nuclear norm of $S^{*}$ and the latticially $p$-nuclear norm of $S$ coincide. Finally, we use our representation theorem for $(\widetilde{\mathcal{N}}^{p}(X,E))^{*}$ to
describe the space of positive $p$-majorizing operators via positively $p$-nuclear operators and nuclear operators.

The operator ideal of $p$-integral operators is defined to be the maximal hull of the ideal of $p$-nuclear operators (\cite{P}). Following A. Defant and K. Floret \cite{DF}, the maximal hull of a Banach operator ideal is defined by finite dimensional subspaces and finite co-dimensional subspaces. It should be mentioned that the maximal hull can be restated by finite rank operators (see \cite[Theorem 8.7.4]{P}). Based on this restatement, O. I. Zhukova \cite{Zhu} defined the class of latticially $p$-integral operators to be the left positive maximal hull of the class of latticially $p$-nuclear operators. In Section 3, we define the class of positively $p$-integral operators to be the right positive maximal hull of the class of positively $p$-nuclear operators. Relating to order completeness, we show that positively $p$-integral operators can be characterized by finite dimensional sublattices and finite co-dimensional subspaces. But, when relating to positive metric approximation property, we characterize positively $p$-integral operators only by finite co-dimensional subspaces and latticially $p$-integral operators only by finite dimensional subspaces. As applications, we establish the duality relationships between latticially $p$-integral operators and positively $p$-integral operators. Consequently, we prove that an operator $S:E\rightarrow X$ is latticially $p$-integral precisely when $S^{**}$ is if $X^{**}$ has the positive metric approximation property (resp. an operator $T:X\rightarrow E$ is positively $p$-integral precisely when $T^{**}$ is if $X^{*}$ has the positive metric approximation property). O. I. Zhukova \cite{Zhu} proved that the class of latticially $p$-nuclear operators from $E$ to $X$ can be embedded isometrically into the class of latticially $p$-integral operators from $E$ to $X$ whenever $E^{*}$ has the metric approximation property and $X$ has the positive metric approximation property. Analogously, we prove that the class of positively $p$-nuclear operators from $X$ to $E$ can be embedded isometrically into the class of positively $p$-integral operators from $X$ to $E$ if $X^{*}$ has the positive metric approximation property and $E$ has the metric approximation property. \cite[Theorem 19.2.13]{P} stated that the adjoint operator ideal $[\prod_{p},\pi_{p}]^{*}$ of $[\prod_{p},\pi_{p}]$ is $[\mathcal{I}_{p^{*}},i_{p^{*}}]$. This formula described $\mathcal{I}_{p^{*}}(F,E^{**})$ as the dual of the $\pi_{p}$-closure of $\mathcal{F}(E,F)$ in $\prod_{p}(E,F)$ when $E^{*}$ and $F$ has the metric approximation property. Analogously, O. I. Zhukova \cite{Zhu} described $\widetilde{\mathcal{I}}_{p^{*}}(E,X^{**})$, the space of latticially $p^{*}$-integral operators from $E$ to $X^{**}$, as the dual of the $\|\cdot\|_{\Lambda_{p}}$-closure of $\mathcal{F}(X,E)$ in the space of latticially $p$-summing operators if $E^{*}$ has the metric approximation property and $X^{**}$ has the positive metric approximation property. In this section, we describe $\widetilde{\mathcal{I}}^{p^{*}}(X,E^{**})$, the space of positively $p^{*}$-integral operators from $X$ to $E^{**}$, as the dual of the $\|\cdot\|_{\Upsilon_{p}}$-closure of $\mathcal{F}(E,X)$ in the space of positive $p$-majorizing operators if $E^{**}$ has the metric approximation property, $X^{*}$ has the positive metric approximation property and $X$ is order continuous.

\noindent {\bf Notation and Preliminary.} Our notation and terminology are standard as may be found in [28,10,23]. Throughout the paper, $X,Y,Z$ will always denote real Banach lattices, whereas $E,F,G$ will denote real Banach spaces. By an operator, we always mean a bounded linear operator.
For a Banach lattice $X$, we denote by $X_{+}$ the positive cone of $X$, i.e., $X_{+}:=\{x\in X: x\geq 0\}$.
We write $LDim(X)$ for the collection of all finite dimensional sublattices of $X$.
If $M$ is a closed subspace of $E$, we denote by $i_{M}$ the canonical inclusion from $M$ into $E$ and by $Q_{M}$ the natural quotient map from $E$ onto $E/M$.
We let $M^{\perp}:=\{u^{*}\in E^{*}: \langle u^{*},u\rangle=0$ for all $u\in M\}$.
We write $FIN(E)$ for the collection of all finite-dimensional subspaces of $E$ and $COFIN(E)$ for the collection of all finite co-dimensional subspaces of $E$. An operator $T:X\rightarrow Y$ which preserves the lattice operations
is called \textit{lattice homomorphism}, that is, $T(x_{1}\vee x_{2})=Tx_{1}\vee Tx_{2}$ for all $x_{1},x_{2}\in X$.
An one-to-one, surjective lattice homomorphism is called \textit{lattice isomorphism}.
As customary, $B_{E}$ denotes the closed unit ball of $E$, $E^{*}$ its linear dual and $I_{E}$ the identity map on $E$. We denote by $\mathcal{L}(E,F)$ (resp. $\mathcal{F}(E,F)$) the space of all operators (resp. finite rank operators) from $E$ to $F$.
The classes of $p$-summing, $p$-nuclear and $p$-integral operators are denoted by $\prod_{p}, \mathcal{N}_{p}$ and $\mathcal{I}_{p}$, respectively. For Banach lattices $X$ and $Y$, $\mathcal{F}_{+}(X,Y)$ stands for the set of all positive finite rank operators from $X$ to $Y$.
The letters $p,q,r$ will designate elements of $[1,+\infty]$, and $p^{*}$ denotes the exponent conjugate to $p$ (i.e., $\frac{1}{p}+\frac{1}{p^{*}}=1$).
For a Banach space $E$, we denote by $l_{p}(E)$ and $l^{w}_{p}(E)$ the spaces of all $p$-summable and weakly $p$-summable sequences in $E$, respectively, with their usual norms
$$\|(u_{n})_{n}\| _{p}:=(\sum_{n=1}^{\infty}\|u_{n}\|^{p})^{\frac{1}{p}}, \quad \|(u_{n})_{n}\|_{p}^{w}:=\sup_{u^{*}\in B_{E^{*}}}(\sum_{n=1}^{\infty}|\langle u^{*},u_{n}\rangle|^{p})^{\frac{1}{p}}.$$

The reader is referred to [28,10,23] for any unexplained notation or terminology.

\section{Positively $p$-nuclear operators}

Recall \cite{P} that an operator $S:E\rightarrow F$ is called \textit{$p$-nuclear} if $$S=\sum_{j=1}^{\infty}u^{*}_{j}\otimes v_{j},$$
where $(u^{*}_{j})_{j}\in l_{p}(E^{*}), (v_{j})_{j}\in l^{w}_{p^{*}}(F)$.

One set $$\nu_{p}(S):=\inf \|(u^{*}_{j})_{j}\|_{p}\cdot\|(v_{j})_{j}\|_{p^{*}}^{w},$$
where the infimum is taken over all so-called \textit{$p$-nuclear representations} described above.

$1$-nuclear operators are simply called \textit{nuclear operators}. The class of all nuclear operators with nuclear norm is denoted by $[\mathcal{N},\nu].$
O. I. Zhukova \cite{Zhu} introduced the concept of latticially $p$-nuclear operators which can be considered to be partially positive analogues of $p$-nuclear operators as follows.

\begin{definition}\cite{Zhu}
An operator $S:E\rightarrow X$ is called \textit{latticially $p$-nuclear} if
\begin{align}\label{6}
S=\sum_{j=1}^{\infty}u^{*}_{j}\otimes x_{j},
\end{align}
where $(u^{*}_{j})_{j}\in l_{p}(E^{*}), (x_{j})_{j}\in l^{w}_{p^{*}}(X)_{+}$.

The representation (\ref{6}) is referred to as a \textit{latticially $p$-nuclear representation} of $S$.

Put $$\widetilde{\nu}_{p}(S):=\inf \|(u^{*}_{j})_{j}\|_{p}\cdot\|(x_{j})_{j}\|_{p^{*}}^{w},$$
where the infimum is taken over all latticially $p$-nuclear representations of $S$.
\end{definition}
The class of all latticially $p$-nuclear operators is denoted by $\widetilde{\mathcal{N}}_{p}$. O. I. Zhukova \cite{Zhu} observed that
latticially $p$-nuclear operators have the left positive ideal property, that is, if $S\in \widetilde{\mathcal{N}}_{p}(E,X), T\in \mathcal{L}(F,E)$ and $R:X\rightarrow Y$ is positive, then $RST$ is latticially $p$-nuclear and $\widetilde{\nu}_{p}(RST)\leq \|R\|\widetilde{\nu}_{p}(S)\|T\|$.
It was also pointed out in \cite{Zhu} that $[\widetilde{\mathcal{N}}_{p},\widetilde{\nu}_{p}]\subseteq [\widetilde{\mathcal{N}}_{q},\widetilde{\nu}_{q}]$ for $p<q$.
O. I. Zhukova \cite{Zhu} mentioned that an operator $S:E\rightarrow X$ is latticially $p$-nuclear if and only if
\begin{align}\label{10}
S=\sum_{j=1}^{\infty}u^{*}_{j}\otimes x_{j},
\end{align}
where $(u^{*}_{j})_{j}\in l_{p}(E^{*}), (|x_{j}|)_{j}\in l^{w}_{p^{*}}(X)$.

O. I. Zhukova set $$\widetilde{\nu}_{p}'(S):=\inf \|(u^{*}_{j})_{j}\|_{p}\cdot\|(|x_{j}|)_{j}\|_{p^{*}}^{w},$$
where the infimum is taken over all  representations (\ref{10}) of $S$.

He also observed that $\widetilde{\nu}_{p}'\leq \widetilde{\nu}_{p}\leq 2\widetilde{\nu}_{p}'$ and
$[\widetilde{\mathcal{N}}_{1},\widetilde{\nu}_{1}']=[\mathcal{N},\nu].$

Recall (\cite{Per},\cite[Sec.6.2]{R}) that an operator $S:E\rightarrow F$ is called \textit{right $p$-nuclear} if $S$ can be written as
$$S=\sum_{j=1}^{\infty}u^{*}_{j}\otimes v_{j},$$

\noindent where $(u^{*}_{j})_{j}\in l^{w}_{p^{*}}(E^{*}), (v_{j})_{j}\in l_{p}(F)$.
Moreover, the right $p$-nuclear norm of $S$ is defined as
$$\nu^{p}(S):=\inf \|(u^{*}_{j})_{j}\|_{p^{*}}^{w}\cdot\|(v_{j})_{j}\|_{p},$$
where the infimum is taken all over possible representations of $S$ as above.
The class of all right $p$-nuclear operators is denoted by $\mathcal{N}^{p}$. It is easy to see that if $S:E\rightarrow F$ is $p$-nuclear, then $S^{*}$ is right $p$-nuclear and $\nu^{p}(S^{*})\leq \nu_{p}(S)$.

In this section, we introduce the notion of positively $p$-nuclear operators, inspired by the definition in the Banach space setting.

\begin{definition}
We say that an operator $T: X\rightarrow E$ is \textit{positively $p$-nuclear} if
\begin{align}\label{7}
T=\sum_{j=1}^{\infty}x^{*}_{j}\otimes u_{j},
\end{align}
where $(x^{*}_{j})_{j}\in l^{w}_{p^{*}}(X^{*})_{+}, (u_{j})_{j}\in l_{p}(E)$.
We call the representation (\ref{7}) a \textit{positively $p$-nuclear representation} of $T$. We set
$$\widetilde{\nu}^{p}(T):=\inf \|(x^{*}_{j})_{j}\|_{p^{*}}^{w}\cdot\|(u_{j})_{j}\|_{p},$$
where the infimum is taken over all positively $p$-nuclear representations of $T$.
The class of all positively $p$-nuclear operators is denoted by $\widetilde{\mathcal{N}}^{p}$.
\end{definition}

We collect some basic properties of positively $p$-nuclear operators which are immediate from the definition. These elementary properties will be used throughout the paper.

\begin{proposition}\label{1.1}
\item[(a)]If $T\in \widetilde{\mathcal{N}}^{p}(X,E), S\in \mathcal{L}(E,F)$ and $R:Y\rightarrow X$ is positive, then $STR$ is positively $p$-nuclear and $\widetilde{\nu}^{p}(STR)\leq \|S\|\widetilde{\nu}^{p}(T)\|R\|$.
\item[(b)]$[\widetilde{\mathcal{N}}^{p},\widetilde{\nu}^{p}]\subseteq [\widetilde{\mathcal{N}}^{q}, \widetilde{\nu}^{q}] $ for $p<q$.
\item[(c)]$T: X\rightarrow E$ is positively $p$-nuclear if and only if
\begin{align}\label{8}
T=\sum_{j=1}^{\infty}x^{*}_{j}\otimes u_{j},
\end{align}
where $(|x^{*}_{j}|)_{j}\in l^{w}_{p^{*}}(X^{*}), (u_{j})_{j}\in l_{p}(E)$. In this case, if we let $$|\widetilde{\nu}^{p}|(T):=\inf \|(|x^{*}_{j}|)_{j}\|_{p^{*}}^{w}\cdot\|(u_{j})_{j}\|_{p},$$
where the infimum is taken over all representations (\ref{8}) of $T$, then $$|\widetilde{\nu}^{p}|(T)\leq \widetilde{\nu}^{p}(T)\leq 2|\widetilde{\nu}^{p}|(T).$$
\item[(d)]$[\widetilde{\mathcal{N}}^{1},|\widetilde{\nu}^{1}|]=[\mathcal{N},\nu].$
\item[(e)]If $S\in \widetilde{\mathcal{N}}_{p}(E,X)$, then $S^{*}\in \widetilde{\mathcal{N}}^{p}(X^{*},E^{*})$ and $\widetilde{\nu}^{p}(S^{*})\leq
\widetilde{\nu}_{p}(S).$ The converse is true if $X$ is a dual Banach lattice and $\widetilde{\nu}^{p}(S^{*})=\widetilde{\nu}_{p}(S).$
\item[(f)]If $T\in \widetilde{\mathcal{N}}^{p}(X,E)$, then $T^{*}\in \widetilde{\mathcal{N}}_{p}(E^{*},X^{*})$ and $\widetilde{\nu}_{p}(T^{*})\leq
\widetilde{\nu}^{p}(T).$ The converse is true if $E$ is a dual Banach space and $\widetilde{\nu}_{p}(T^{*})=\widetilde{\nu}^{p}(T).$
\end{proposition}

\begin{remark}
The class $\widetilde{\mathcal{N}}^{p}$ do not coincide with $\mathcal{N}^{p}$. Indeed, O. I. Zhukova \cite{Zhu} remarked that the operator $T:L_{1}[0,1]\rightarrow L_{2}[0,1]$ defined by $$Tf=\sum\limits_{n=1}^{\infty}\frac{1}{n}(\int_{[0,1]}f(t)r_{n}(t)dt)r_{n},\quad f\in L_{1}[0,1],$$
where $(r_{n})_{n}$ is the Rademacher function sequence, being $p$-nuclear for every $p>1$, is not latticially $p$-nuclear for any $p$. Hence, $T^{*}$ is right $p$-nuclear for every $p>1$. But, by Proposition \ref{1.1} (e), $T^{*}$ is not positively $p$-nuclear for any $p$.
\end{remark}

To describe the conjugate of the space of positively $p$-nuclear operators, we need the concept of positive $p$-majorizing operators introduced in \cite{CBC}.

\begin{definition}\label{1.80}\cite{CBC}
We say that an operator $S:E\rightarrow X$ is \textit{positive $p$-majorizing} if there exists a constant $C>0$ such
\begin{equation}\label{109}
(\sum_{j=1}^{n}|\langle x^{*}_{j},Su_{j}\rangle|^{p})^{\frac{1}{p}}\leq C \|(x^{*}_{j})_{j=1}^{n}\|^{w}_{p},
\end{equation}
for all finite families $(u_{j})_{j=1}^{n}$ in $B_{E}$ and $(x^{*}_{j})_{j=1}^{n}$ in $(X^{*})_{+}$.
\end{definition}

We denote by $\Upsilon_{p}(E,X)$ the space of all positive $p$-majorizing operators from $E$ to $X$. It is easy to see that $\Upsilon_{p}(E,X)$ becomes a Banach space with the norm $\|\cdot\|_{\Upsilon_{p}}$ given by the infimum of the constants $C$ satisfying (\ref{109}). Obviously,
positive $p$-majorizing operators have the left positive ideal property, that is, if $S\in \Upsilon_{p}(E,X), T\in \mathcal{L}(F,E)$ and $R:X\rightarrow Y$ is positive, then $RST$ is positive $p$-majorizing and $\|RST\|_{\Upsilon_{p}}\leq \|R\|\|S\|_{\Upsilon_{p}}\|T\|$.

\begin{definition}\cite{B2}\label{1.9}
An operator $T:X\rightarrow E$ is said to be \textit{positive $p$-summing} if there exists a constant $C>0$ such that
\begin{equation}\label{1}
(\sum_{i=1}^{n}\|Tx_{i}\|^{p})^{\frac{1}{p}}\leq C \|(x_{i})_{i=1}^{n}\|^{w}_{p}.
\end{equation}
for any choice of finitely many vectors $x_{1}, x_{2},\cdots, x_{n}$ in $X_{+}$.
\end{definition}
The space of all positive $p$-summing operators from $X$ to $E$ is denoted by $\Lambda_{p}(X,E)$. This space becomes a Banach space with the norm $\|\cdot\|_{\Lambda_{p}}$ given by the infimum of the constants $C$ satisfying (\ref{1}).
It is easy to see that positive $p$-summing operators have the right positive ideal property, that is, if $T\in \Lambda_{p}(X,E), S\in \mathcal{L}(E,F)$ and $R:Y\rightarrow X$ is positive, then $STR$ is positive $p$-summing and $\|STR\|_{\Lambda_{p}}\leq \|S\|\|T\|_{\Lambda_{p}}\|R\|$.

In \cite{CBC}, we prove the following duality relationships between positive $p$-summing operators and positive $p$-majorizing operators which will be used later.

\begin{theorem}\label{1.6}\cite{CBC}
\item[(a)]An operator $T:X\rightarrow E$ is positive $p$-summing if and only if $T^{*}$ is positive $p$-majorizing. In this case, $\|T\|_{\Lambda_{p}}=\|T^{*}\|_{\Upsilon_{p}}$.
\item[(b)]An operator $S:F\rightarrow Y$ is positive $p$-majorizing if and only if $S^{*}$ is positive $p$-summing. In this case, $\|S\|_{\Upsilon_{p}}=\|S^{*}\|_{\Lambda_{p}}$.
\end{theorem}

Recall that a Banach space $E$ has the \textit{approximation property} ($AP$ for short) if for every $\epsilon>0$ and for every compact subset $K$ of $E$, there exists an operator $S\in \mathcal{F}(E)$ such that $\|Su-u\|<\epsilon$ for all $u\in K$. In addition, if the operator $S$ can be chosen with $\|S\|\leq 1$, $E$ is said to has the \textit{metric approximation property} ($MAP$).
A Banach lattice $X$ is said to have the \textit{positive metric approximation property} ($PMAP$) if for every $\epsilon>0$ and for every compact subset $K$ of $X$, there exists an operator $R\in \mathcal{F}_{+}(X)$ with $\|R\|\leq 1$ such that $\|Rx-x\|<\epsilon$ for all $x\in K$.

We need a result due to A. Lissitsin and E. Oja \cite{LO} which says that positive finite-rank operators between dual Banach lattices are locally conjugate.

\begin{lemma}\label{3.8}\cite{LO}
Let $X,Y$ be Banach lattices, let $F$ be a finite subset of $Y^{*}$ and let $\epsilon>0$. If $S\in \mathcal{F}_{+}(Y^{*},X^{*})$, then there exists an operator $R\in \mathcal{F}_{+}(X,Y)$ such that $\|R\|\leq (1+\epsilon)\|S\|$ and $\|R^{*}y^{*}-Sy^{*}\|<\epsilon$ for all $y^{*}\in F$.
\end{lemma}

It follows from Lemma \ref{3.8} that $X^{*}$ has the $PMAP$ if and only if for every $\epsilon>0$ and each compact subset $K$ of $X^{*}$, there exists an operator
$R\in \mathcal{F}_{+}(X)$ with $\|R\|\leq 1$ such that $\|R^{*}x^{*}-x^{*}\|<\epsilon$ for all $x^{*}\in K$.

\begin{lemma}\label{1.2}
Suppose that $E$ has the $AP$ or $X^{*}$ has the $PMAP$. Assume that $S:E\rightarrow X^{**}$ is positive $p^{*}$-majorizing. Let $(x^{*}_{n})_{n}\in (l^{w}_{p^{*}}(X^{*}))_{+}$ and $(u_{n})_{n}\in l_{p}(E)$. Then
\begin{center}
$\sum\limits_{n=1}^{\infty}x^{*}_{n}\otimes u_{n}=0$ implies $\sum\limits_{n=1}^{\infty}\langle Su_{n},x^{*}_{n}\rangle=0.$
\end{center}
\end{lemma}
\begin{proof}
It is clear that the conclusion holds true if $S$ is finite-rank. Suppose that $S:E\rightarrow X^{**}$ is positive $p^{*}$-majorizing.

\noindent Case 1. $E$ has the $AP$.

Let $\epsilon>0$. Choose $1\leq \xi_{n}\rightarrow \infty$ with $\|(\xi_{n}u_{n})_{n}\|_{p}\leq (1+\epsilon)\|(u_{n})_{n}\|_{p}$. Since $E$ has the $AP$, there exists an operator $U\in \mathcal{F}(E)$ such that $\|U(\frac{u_{n}}{\xi_{n}\|u_{n}\|})-\frac{u_{n}}{\xi_{n}\|u_{n}\|}\|<\epsilon$ for all $n$. Note that $\sum\limits_{n=1}^{\infty}\langle SUu_{n},x^{*}_{n}\rangle=0$. By Theorem \ref{1.6}, we get
\begin{align*}
|\sum\limits_{n=1}^{\infty}\langle Su_{n},x^{*}_{n}\rangle|&=|\sum\limits_{n=1}^{\infty}\langle S(u_{n}-Uu_{n}),x^{*}_{n}\rangle|\\
&=|\sum\limits_{n=1}^{\infty}\langle S^{*}J_{X^{*}}x^{*}_{n},u_{n}-Uu_{n}\rangle|\\
&\leq (\sum_{n=1}^{\infty}\|S^{*}J_{X^{*}}x^{*}_{n}\|^{p^{*}})^{\frac{1}{p^{*}}}(\sum_{n=1}^{\infty}\|u_{n}-Uu_{n}\|^{p})^{\frac{1}{p}}\\
&\leq \epsilon(1+\epsilon)\|S\|_{\Upsilon_{p^{*}}}\|(x^{*}_{n})_{n=1}^{\infty}\|_{p^{*}}^{w}\|(u_{n})_{n}\|_{p}\\
\end{align*}

Letting $\epsilon\rightarrow 0$, we get $$\sum\limits_{n=1}^{\infty}\langle Su_{n},x^{*}_{n}\rangle=0.$$

\noindent Case 2. $X^{*}$ has the $PMAP$.

We may assume that $\lim\limits_{n\rightarrow \infty}\|x^{*}_{n}\|=0$. Let $\epsilon>0$. We choose a positive integral $N$ with $(\sum\limits_{n=N+1}^{\infty}\|u_{n}\|^{p})^{\frac{1}{p}}<\epsilon.$ Let $\delta>0$ be such that
$\delta N^{\frac{1}{p^{*}}}\|S\|\|(u_{n})_{n=1}^{\infty}\|_{p}<\epsilon.$ Since $X^{*}$ has the $PMAP$, it follows from Lemma \ref{3.8} that there exists an operator $R\in \mathcal{F}_{+}(X)$ with $\|R\|\leq 1$ such that $\|R^{*}x^{*}_{n}-x^{*}_{n}\|<\delta$ for all $n$. Note that $\sum\limits_{n=1}^{\infty}\langle R^{**}Su_{n},x^{*}_{n}\rangle=0.$ By Theorem \ref{1.6}, we get
\begin{align*}
|\sum\limits_{n=1}^{\infty}\langle Su_{n},x^{*}_{n}\rangle|&=|\sum\limits_{n=1}^{\infty}\langle Su_{n},x^{*}_{n}-R^{*}x^{*}_{n}\rangle|\\
&\leq \sum\limits_{n=1}^{N}|\langle Su_{n},x^{*}_{n}-R^{*}x^{*}_{n}\rangle|+\sum\limits_{n=N+1}^{\infty}|\langle S^{*}x^{*}_{n},u_{n}\rangle|+
\sum\limits_{n=N+1}^{\infty}|\langle S^{*}R^{*}x^{*}_{n},u_{n}\rangle|\\
&\leq (\sum_{n=1}^{N}\|x^{*}_{n}-R^{*}x^{*}_{n}\|^{p^{*}})^{\frac{1}{p^{*}}}(\sum_{n=1}^{N}\|Su_{n}\|^{p})^{\frac{1}{p}}+
(\sum_{n=N+1}^{\infty}\|S^{*}x^{*}_{n}\|^{p^{*}})^{\frac{1}{p^{*}}}(\sum_{n=N+1}^{\infty}\|u_{n}\|^{p})^{\frac{1}{p}}\\
&+(\sum_{n=N+1}^{\infty}\|S^{*}R^{*}x^{*}_{n}\|^{p^{*}})^{\frac{1}{p^{*}}}(\sum_{n=N+1}^{\infty}\|u_{n}\|^{p})^{\frac{1}{p}}\\
&\leq \delta N^{\frac{1}{p^{*}}}\|S\|\|(u_{n})_{n=1}^{\infty}\|_{p}+2\epsilon\|S\|_{\Upsilon_{p^{*}}}\|(x^{*}_{n})_{n=1}^{\infty}\|_{p^{*}}^{w}\\
&\leq \epsilon+2\epsilon\|S\|_{\Upsilon_{p^{*}}}\|(x^{*}_{n})_{n=1}^{\infty}\|_{p^{*}}^{w}\\
\end{align*}
Letting $\epsilon\rightarrow 0$, we get $$\sum\limits_{n=1}^{\infty}\langle Su_{n},x^{*}_{n}\rangle=0.$$

\end{proof}

Consequently, under the hypothesis of Lemma \ref{1.2}, if $T:X\rightarrow E$ is positively $p$-nuclear and $S:E\rightarrow X^{**}$ is positive $p^{*}$-majorizing, then $\textrm{trace}(ST):=\sum\limits_{n=1}^{\infty}\langle Su_{n},x^{*}_{n}\rangle$ is independent of the choice of the positively $p$-nuclear representation $T=\sum\limits_{n=1}^{\infty}x^{*}_{n}\otimes u_{n}$. Moreover, it is easy to see that $|\textrm{trace}(ST)|\leq \|S\|_{\Upsilon_{p^{*}}}\widetilde{\nu}^{p}(T).$

To prove the main result of this section, we need a lemma due to A. Lissitsin and E. Oja \cite{LO} that demonstrates the connection between finite-dimensional subspaces and finite-dimensional sublattices in order complete Banach lattices. This lemma will be used frequently throughout this paper.

\begin{lemma}\label{1.4}\cite[Lemma 5.5]{LO}
Let $M$ be a finite-dimensional subspace of an order complete Banach lattice $X$ and let
$\epsilon > 0$. Then there exist a sublattice $Z$ of $X$ containing $M$, a finite-dimensional sublattice $G$ of $Z$, and a positive projection $P$ from $Z$ onto $G$ such that $\|Px-x\|\leq \epsilon\|x\|$ for all $x\in M$.
\end{lemma}

We also need the principle of local reflexivity in Banach lattices due to J. L. Conroy, L. C. Moore \cite{CM} and S. J. Bernau \cite{Be}, which plays a crucial role in Banach lattice theory.

\begin{theorem}\label{1.3}\cite[Theorem 2]{Be}
Let $X$ be a Banach lattice and let $M$ be a finite-dimensional sublattice of $X^{**}$.
Then for every finite-dimensional subspace $L$ of $X^{*}$ and  every $\epsilon > 0$,
there exists a lattice isomorphism $R$ from $M$ into $X$ such that
\item[(i)]$\|R\|, \|R^{-1}\| \leq 1+\epsilon$;
\item[(ii)]$|\langle x^{**},x^*\rangle-\langle x^{*},Rx^{**}\rangle| \leq \epsilon\|x^{**}\|\|x^*\|$,
for all $x^{**} \in M$ and $x^{*} \in L$.
\end{theorem}

Now we are in a position to give the main result of this section.

\begin{theorem}\label{1.13}
Suppose that $E$ has the $AP$ or $X^{*}$ has the $PMAP$. Then $$\Upsilon_{p^{*}}(E,X^{**})=(\widetilde{\mathcal{N}}^{p}(X,E))^{*}.$$
\end{theorem}
\begin{proof}
Let us define an operator $$V:\Upsilon_{p^{*}}(E,X^{**})\rightarrow(\widetilde{\mathcal{N}}^{p}(X,E))^{*}$$ by $$S\mapsto V_{S}(T)=\textrm{trace}(ST), \quad S \in\Upsilon_{p^{*}}(E,X^{**}), T\in \widetilde{\mathcal{N}}^{p}(X,E).$$
Then $\|V_{S}\|\leq \|S\|_{\Upsilon_{p^{*}}}$.

Let $\varphi\in (\widetilde{\mathcal{N}}^{p}(X,E))^{*}$. We define an operator $S:E\rightarrow X^{**}$ by $\langle Su,x^{*}\rangle=\langle\varphi, x^{*}\otimes u\rangle$ for $u\in E, x^{*}\in X^{*}$.

We claim that $S$ is positive $p^{*}$-majorizing.

Given any $u_{1},u_{2},\cdots,u_{n}$ in $B_{E}$ and $x^{***}_{1},x^{***}_{2},\cdots,x^{***}_{n}$ in $(X^{***})_{+}$. Let $\epsilon>0$. We set $M=\textrm{span}\{Su_{j}:1\leq j\leq n\}$ and $L=\textrm{span}\{x^{***}_{j}:1\leq j\leq n\}$. It follows from Lemma \ref{1.4} that there exist a sublattice $Z$ of $X^{***}$ containing $L$, a finite-dimensional sublattice $G$ of $Z$ and a positive projection $P$ from $Z$ onto $G$ such that $\|Px^{***}-x^{***}\|\leq \epsilon\|x^{***}\|$ for all $x^{***}\in L$. By Theorem \ref{1.3}, we get a lattice isomorphism $R$ from $G$ into $X^{*}$ such that
$\|R\|, \|R^{-1}\| \leq 1+\epsilon$ and

\begin{align}\label{4}
|\langle x^{***},x^{**}\rangle-\langle x^{**},Rx^{***}\rangle| \leq \epsilon\|x^{***}\|\|x^{**}\|,
\end{align}
for all $x^{***} \in G, x^{**} \in M.$
Let $x^{*}_{j}=RPx^{***}_{j}\geq 0(j=1,2,\cdots,n)$. We choose $(\lambda_{j})_{j=1}^{n}$ such that $\sum\limits_{j=1}^{n}|\lambda_{j}|^{p}=1$ and

$$(\sum_{j=1}^{n}|\langle Su_{j},x^{*}_{j}\rangle|^{p^{*}})^{\frac{1}{p^{*}}}=\sum_{j=1}^{n}\lambda_{j}\langle Su_{j},x^{*}_{j}\rangle.$$

Let $T=\sum\limits_{j=1}^{n}x^{*}_{j}\otimes \lambda_{j}u_{j}\in \widetilde{\mathcal{N}}^{p}(X,E)$. Then we have
\begin{align}\label{3}
(\sum_{j=1}^{n}|\langle Su_{j},x^{*}_{j}\rangle|^{p^{*}})^{\frac{1}{p^{*}}}&=\langle \varphi,T\rangle \nonumber \\
&\leq \|\varphi\|\widetilde{\nu}^{p}(T) \nonumber \\
&\leq \|\varphi\|\|(x^{*}_{j})_{j=1}^{n}\|^{w}_{p^{*}} \nonumber \\
&\leq \|\varphi\|(1+\epsilon)^{2}\|(x^{***}_{j})_{j=1}^{n}\|^{w}_{p^{*}}
\end{align}
By (\ref{4}), we get

\begin{align}\label{5}
(\sum_{j=1}^{n}|\langle x^{***}_{j},Su_{j}\rangle-\langle Su_{j},x^{*}_{j}\rangle|^{p^{*}})^{\frac{1}{p^{*}}}&\leq
(\sum_{j=1}^{n}|\langle x^{***}_{j},Su_{j}\rangle-\langle Px^{***}_{j},Su_{j}\rangle|^{p^{*}})^{\frac{1}{p^{*}}} \nonumber \\
&+(\sum_{j=1}^{n}|\langle Px^{***}_{j},Su_{j}\rangle-\langle Su_{j},x^{*}_{j}\rangle|^{p^{*}})^{\frac{1}{p^{*}}} \nonumber \\
&\leq \epsilon \|S\|(\sum_{j=1}^{n}\|x^{***}_{j}\|^{p^{*}})^{\frac{1}{p^{*}}}+\epsilon (1+\epsilon)\|S\|(\sum_{j=1}^{n}\|x^{***}_{j}\|^{p^{*}})^{\frac{1}{p^{*}}}
\end{align}

Combining (\ref{3}) and (\ref{5}), we get
\begin{align*}
(\sum_{j=1}^{n}|\langle x^{***}_{j},Su_{j}\rangle|^{p^{*}})^{\frac{1}{p^{*}}}
&\leq (\sum_{j=1}^{n}|\langle Su_{j},x^{*}_{j}\rangle|^{p^{*}})^{\frac{1}{p^{*}}}+(\sum_{j=1}^{n}|\langle x^{***}_{j},Su_{j}\rangle-\langle Su_{j},x^{*}_{j}\rangle|^{p^{*}})^{\frac{1}{p^{*}}}\\
&\leq \|\varphi\|(1+\epsilon)^{2}\|(x^{***}_{j})_{j=1}^{n}\|^{w}_{p^{*}}+\epsilon (2+\epsilon)\|S\|(\sum_{j=1}^{n}\|x^{***}_{j}\|^{p^{*}})^{\frac{1}{p^{*}}}
\end{align*}
Letting $\epsilon\rightarrow 0$, we get $$(\sum_{j=1}^{n}|\langle x^{***}_{j},Su_{j}\rangle|^{p^{*}})^{\frac{1}{p^{*}}}\leq \|\varphi\|\|(x^{***}_{j})_{j=1}^{n}\|^{w}_{p^{*}},$$ which implies that $S$ is positive $p^{*}$-majorizing and $\|S\|_{\Upsilon_{p^{*}}}\leq \|\varphi\|$.

By the definition of the operator $S$, we see that $\langle\varphi,T\rangle=V_{S}(T)$ for all $T\in \mathcal{F}(X,E)$. Since $\mathcal{F}(X,E)$ is $\widetilde{\nu}^{p}$-dense in $\widetilde{\mathcal{N}}^{p}(X,E)$, it follows that $\varphi=V_{S}$. Hence the operator $V$ is a surjective linear isometry.

\end{proof}

\begin{corollary}
Suppose that $E^{***}$ has the $AP$ or $X^{*}$ has the $PMAP$. If the operator $T:X\rightarrow E$ has a latticially $p$-nuclear adjoint, then $T$ is positively $p$-nuclear and $\widetilde{\nu}^{p}(T)=\widetilde{\nu}_{p}(T^{*}).$
\end{corollary}
\begin{proof}
Suppose that $T$ is not positively $p$-nuclear. Since $T^{*}$ is latticially $p$-nuclear, it follows from Proposition \ref{1.1}(e) that $T^{**}$ is positively $p$-nuclear and so is $T^{**}J_{X}=J_{E}T$. Hence $J_{E}T\in \widetilde{\mathcal{N}}^{p}(X,E^{**})\setminus \widetilde{\mathcal{N}}^{p}(X,E).$
Since $E^{***}$ has the $AP$, $E^{**}$ has the $AP$.
By Theorem \ref{1.13} and the Hahn-Banach Theorem,  we get an operator $S\in \Upsilon_{p^{*}}(E^{**},X^{**})$ such that $\textrm{trace}(SJ_{E}T)=1$ and $\textrm{trace}(SJ_{E}R)=0$ for all $R\in \widetilde{\mathcal{N}}^{p}(X,E).$ This yields that $SJ_{E}u=0$ for all $u\in E$. Let us take any latticially $p$-nuclear representation $T^{*}=\sum\limits_{n=1}^{\infty}u^{**}_{n}\otimes x^{*}_{n}$. Since $SJ_{E}u=0$ for all $u\in E$, we get $\sum\limits_{n=1}^{\infty}\langle x^{*}_{n},x\rangle Su^{**}_{n}=0$ for all $x\in X$. Moreover, for every $x^{*}\in X^{*}$, we have
$$0=\langle SJ_{E}u, x^{*}\rangle=\langle J^{*}_{E}S^{*}x^{*},u\rangle.$$
By Goldstine-Weston Theorem, we get
\begin{align}\label{505}
\langle J^{**}_{E}u^{**},S^{*}x^{*}\rangle=\langle u^{**},J^{*}_{E}S^{*}x^{*}\rangle=0.
\end{align}
Note that
\begin{align}\label{506}
1=\textrm{trace}(SJ_{E}T)=\textrm{trace}(ST^{**}J_{X})=\sum\limits_{n=1}^{\infty}\langle Su^{**}_{n},x^{*}_{n}\rangle.
\end{align}
It follows from Theorem \ref{1.6} that
$$\sum\limits_{n=1}^{\infty}\|u^{**}_{n}\|\|S^{*}x^{*}_{n}\|\leq \|(u^{**}_{n})_{n}\|_{p}\|S\|_{\Upsilon_{p^{*}}}\|(x^{*}_{n})_{n}\|_{p^{*}}^{w}<\infty.$$

In the case $X^{*}$ has the $PMAP$, an argument analogous to that of Lemma \ref{1.2} Case 2 shows that
$\sum\limits_{n=1}^{\infty}\langle Su^{**}_{n},x^{*}_{n}\rangle=0,$ which contradicts with (\ref{506}). It remains to prove the conclusion in the case
$E^{***}$ has the $AP$.

We define an operator
$$V:=S^{*}J_{X^{*}}T^{*}J_{E}^{*}: E^{***}\stackrel{J^{*}_{E}}{\longrightarrow}E^{*}\stackrel{T^{*}}{\longrightarrow}X^{*}\stackrel{J_{X^{*}}}{\longrightarrow}X^{***}\stackrel{S^{*}}{\longrightarrow}
E^{***}.$$
It is easy to see that $V=\sum\limits_{n=1}^{\infty}u^{**}_{n}\otimes S^{*}x^{*}_{n}$ is nuclear. Furthermore, for $u^{***}\in E^{***}, u^{**}\in E^{**}$, we get
\begin{align*}
\langle Vu^{***},u^{**}\rangle &=\langle S^{*}J_{X^{*}}T^{*}J_{E}^{*}u^{***},u^{**}\rangle\\
&=\sum\limits_{n=1}^{\infty}\langle u^{**}_{n},J^{*}_{E}u^{***}\rangle \langle S^{*}x^{*}_{n},u^{**}\rangle\\
&=\sum\limits_{n=1}^{\infty}\langle J^{**}_{E}u^{**}_{n},u^{***}\rangle \langle S^{*}x^{*}_{n},u^{**}\rangle.\\
\end{align*}
Therefore, $V=\sum\limits_{n=1}^{\infty}J^{**}_{E}u^{**}_{n}\otimes S^{*}x^{*}_{n}, \sum\limits_{n=1}^{\infty}\|J^{**}_{E}u^{**}_{n}\|\|S^{*}x^{*}_{n}\|<\infty.$
Since $E^{***}$ has the $AP$, we get, by (\ref{506}),
$$\sum\limits_{n=1}^{\infty}\langle J^{**}_{E}u^{**}_{n},S^{*}x^{*}_{n}\rangle=\sum\limits_{n=1}^{\infty}\langle S^{*}x^{*}_{n},u^{**}_{n}\rangle
=\sum\limits_{n=1}^{\infty}\langle Su^{**}_{n},x^{*}_{n}\rangle=1.$$
This contradicts with (\ref{505}).

In conclusion, we have proved in both cases that if $J_{E}T$ is positively $p$-nuclear, then so is $T$. Since $\widetilde{\mathcal{N}}^{p}(X,E)$ is a closed subspace of $\widetilde{\mathcal{N}}^{p}(X,E^{**})$ under the canonical mapping $J_{E}$, we get
$$\widetilde{\nu}^{p}(T)=\widetilde{\nu}^{p}(J_{E}T)=\widetilde{\nu}^{p}(T^{**}J_{X})\leq \widetilde{\nu}^{p}(T^{**})\leq \widetilde{\nu}_{p}(T^{*}).$$

\end{proof}

\begin{theorem}
Suppose that $E^{*}$ has the $AP$ or $X^{****}$ has the $PMAP$. If the operator $S:E\rightarrow X$ has a positively $p$-nuclear adjoint, then $S$ is latticially $p$-nuclear and $\widetilde{\nu}_{p}(S)=\widetilde{\nu}^{p}(S^{*}).$
\end{theorem}
\begin{proof}
Suppose that $S$ is not latticially $p$-nuclear. By \cite[Theorem 3]{Zhu}, there exists an operator $T\in \Lambda_{p^{*}}(X^{**},E^{**})$ such that $\textrm{trace}(TJ_{X}S)=1$ and $\textrm{trace}(TJ_{X}R)=0$ for all $R\in \widetilde{\mathcal{N}}_{p}(E,X)$. This implies that $TJ_{X}x=0$ for all $x\in X$. It follows from Goldstine-Weston Theorem that $\langle J^{**}_{X}x^{**},T^{*}u^{*}\rangle=0$ for all $u^{*}\in E^{*}, x^{**}\in X^{**}$.
Take any positively $p$-nuclear representation $S^{*}=\sum\limits_{n=1}^{\infty}x^{**}_{n}\otimes u^{*}_{n}$. Then, we have
\begin{align}\label{302}
\sum\limits_{n=1}^{\infty}\langle Tx^{**}_{n},u^{*}_{n}\rangle=\textrm{trace}(TS^{**}J_{E})=\textrm{trace}(TJ_{X}S)=1.
\end{align}
Moreover, $\sum\limits_{n=1}^{\infty}\langle u^{*}_{n}, u\rangle Tx^{**}_{n}=0$ for all $u\in E$.

If $E^{*}$ has the $AP$, then $E^{*}$ has the $AP$ with conjugate operators. We argue as in Lemma \ref{1.2} Case 1 to show that
$\sum\limits_{n=1}^{\infty}\langle Tx^{**}_{n},u^{*}_{n}\rangle=0$, which contradicts with (\ref{302}).
Now assume that $X^{****}$ has the $PMAP$.

Let $$U=(T^{*}J_{E^{*}})(S^{*}J_{X}^{*}):X^{***}\stackrel{J^{*}_{X}}{\longrightarrow}X^{*}\stackrel{S^{*}}{\longrightarrow}E^{*}\stackrel{J_{E^{*}}}{\longrightarrow}
E^{***}\stackrel{T^{*}}{\longrightarrow}X^{****}.$$
It is easy to check that $S^{*}J^{*}_{X}=\sum\limits_{n=1}^{\infty}J_{X^{**}}x^{**}_{n}\otimes u^{*}_{n}.$ Combining Lemma \ref{1.2} with Theorem \ref{1.6}, we get
$$0=\sum\limits_{n=1}^{\infty}\langle J_{X}^{**}x^{**}_{n},T^{*}u^{*}_{n}\rangle=\sum\limits_{n=1}^{\infty}\langle J_{X^{**}}x^{**}_{n},T^{*}u^{*}_{n}\rangle
=\sum\limits_{n=1}^{\infty}\langle Tx^{**}_{n},u^{*}_{n}\rangle=1.$$
This is a contradiction.

Therefore, we have proved in both cases that if $J_{X}S$ is latticially $p$-nuclear, so is $S$. Since $\widetilde{\mathcal{N}}_{p}(E,X)$ can be considered to be a closed subspace of $\widetilde{\mathcal{N}}_{p}(E,X^{**})$ under the canonical embedding $J_{X}$, we get
$$\widetilde{\nu}_{p}(S)=\widetilde{\nu}_{p}(J_{X}S)=\widetilde{\nu}_{p}(S^{**}J_{E})\leq \widetilde{\nu}_{p}(S^{**})\leq \widetilde{\nu}^{p}(S^{*}).$$
This completes the proof.
\end{proof}

At the rest of this section, we describe the space of positive $p$-majorizing operators via positively $p$-nuclear operators. First we prove a lemma which is interesting in itself.

\begin{lemma}\label{2.1}
Suppose that $T:X\rightarrow E$ is positively $p$-nuclear and $S:F\rightarrow X$ is positive $p^{*}$-majorizing. Then $TS$ is nuclear and $\nu(TS)\leq \widetilde{\nu}^{p}(T)\|S\|_{\Upsilon_{p^{*}}}$.
\end{lemma}
\begin{proof}
Let $\epsilon>0$. Then $T$ admits a positively $p$-nuclear representation $T=\sum\limits_{n=1}^{\infty}x^{*}_{n}\otimes u_{n}$ such that
$$\|(x^{*}_{n})_{n}\|_{p^{*}}^{w}\|(u_{n})_{n}\|_{p}\leq (1+\epsilon)\widetilde{\nu}^{p}(T).$$
By Theorem \ref{1.6}, we get
\begin{align*}
\sum_{n=1}^{\infty}\|S^{*}x^{*}_{n}\|\|u_{n}\|&\leq (\sum_{n=1}^{\infty}\|S^{*}x^{*}_{n}\|^{p^{*}})^{\frac{1}{p^{*}}}(\sum_{n=1}^{\infty}\|u_{n}\|
^{p})^{\frac{1}{p}}\\
&\leq \|S\|_{\Upsilon_{p^{*}}}\|(x^{*}_{n})_{n}\|_{p^{*}}^{w}\|(u_{n})_{n}\|_{p}\\
&\leq \|S\|_{\Upsilon_{p^{*}}}(1+\epsilon)\widetilde{\nu}^{p}(T).
\end{align*}
This means that $TS$ is nuclear and $\nu(TS)\leq \|S\|_{\Upsilon_{p^{*}}}(1+\epsilon)\widetilde{\nu}^{p}(T).$ Letting $\epsilon\rightarrow 0$, we get
$\nu(TS)\leq \widetilde{\nu}^{p}(T)\|S\|_{\Upsilon_{p^{*}}}$.

\end{proof}

Let $E$ be a Banach space and $X$ be a Banach lattice. We set
\begin{center}
$\mathcal{U}_{*}^{p}(E,X):=\{S\in \mathcal{L}(E,X):TS$ is nuclear for all $T\in \widetilde{\mathcal{N}}^{p}(X,E)\}.$
\end{center}
For $S\in \mathcal{U}_{*}^{p}(E,X)$, we define $$V_{S}:\widetilde{\mathcal{N}}^{p}(X,E)\rightarrow \mathcal{N}(E), T\mapsto TS.$$
It follows from the closed graph theorem that $V_{S}$ is continuous. We define a norm $\zeta^{p}$ on $\mathcal{U}_{*}^{p}(E,X)$ by
$$\zeta^{p}(S):=\|V_{S}\|, \quad S\in \mathcal{U}_{*}^{p}(E,X).$$
A routine argument shows that $[\mathcal{U}_{*}^{p}(E,X),\zeta^{p}]$ is a Banach space.

We note that if $E$ has the $AP$ and $U\in \mathcal{N}(E)$, then $\textrm{trace}(U)=\sum\limits_{n=1}^{\infty}\langle u^{*}_{n},u_{n}\rangle$ is independent of the choice of the nuclear representation $U=\sum\limits_{n=1}^{\infty}u^{*}_{n}\otimes u_{n}$. Moreover, $|\textrm{trace}(U)|\leq \nu(U)$.

\begin{theorem}
Suppose that $E$ has the $AP$. Then $$\Upsilon_{p^{*}}(E,X)=\mathcal{U}_{*}^{p}(E,X)$$ for all Banach lattices $X$.
\end{theorem}
\begin{proof}
By Lemma \ref{2.1}, we get $\Upsilon_{p^{*}}(E,X)\subseteq \mathcal{U}_{*}^{p}(E,X)$ and $\zeta^{p}\leq \|\cdot\|_{\Upsilon_{p^{*}}}$.

Conversely, for $S\in \mathcal{U}_{*}^{p}(E,X)$, we define $\varphi\in (\widetilde{\mathcal{N}}^{p}(X,E))^{*}$ by $$\langle \varphi,T\rangle=\textrm{trace}(TS), \quad T\in\widetilde{\mathcal{N}}^{p}(X,E).$$ Clearly, $\|\varphi\|\leq \zeta^{p}(S)$.
It follows from Theorem \ref{1.13} that there exists a unique operator $\widetilde{S}\in \Upsilon_{p^{*}}(E,X^{**})$ such that $\|\widetilde{S}\|_{\Upsilon_{p^{*}}}=
\|\varphi\|$ and $\textrm{trace}(\widetilde{S}T)=\langle \varphi,T\rangle$ for all $T\in \widetilde{\mathcal{N}}^{p}(X,E)$. The uniqueness of $\widetilde{S}$ implies that $J_{X}S=\widetilde{S}.$ Hence $S$ is positive $p^{*}$-majorizing and
$$\|S\|_{\Upsilon_{p^{*}}}=\|J_{X}S\|_{\Upsilon_{p^{*}}}=\|\varphi\|\leq \zeta^{p}(S).$$
The conclusion follows.

\end{proof}

\section{Positively $p$-integral operators}

Let us begin this section with recalling the definition of maximal Banach operator ideals.

\begin{definition}\cite{DF}\label{1.777}
Let $[\mathfrak{A}, \mathbf{A}]$ be a Banach operator ideal.
\item[(1)] For $T\in \mathcal{L}(E,F)$ define
$$\mathbf{A}^{\max}(T):=\sup\{\mathbf{A}(Q_{L}Ti_{M}): M\in FIN(E), L\in COFIN(F)\}$$
$$\mathfrak{A}^{\max}(E,F):=\{T\in \mathcal{L}(E,F):\mathbf{A}^{\max}(T)<\infty\}$$
and call $[\mathfrak{A},\mathbf{A}]^{\max}:=[\mathfrak{A}^{\max},\mathbf{A}^{\max}]$ the \textit{maximal hull} of $[\mathfrak{A}, \mathbf{A}]$.
\item[(2)] $[\mathfrak{A}, \mathbf{A}]$ is called \textit{maximal} if $[\mathfrak{A}, \mathbf{A}]=[\mathfrak{A}^{\max},\mathbf{A}^{\max}]$.
\end{definition}

There is another criterion for the maximal hull $(\mathfrak{A},\mathbf{A})^{\max}$.

\begin{theorem}\cite{P}\label{1.77}
Let $[\mathfrak{A}, \mathbf{A}]$ be a Banach operator ideal. An operator $T\in \mathcal{L}(E,F)$ belongs to $\mathfrak{A}^{\max}(E,F)$ if and only if there exists a constant $C>0$ such that
\begin{center}
$\mathbf{A}(RTS)\leq C\|R\|\|S\|$ for all $S\in \mathcal{F}(G,E)$ and $R\in \mathcal{F}(F,H),$
\end{center}
where $G,H$ are arbitrary Banach spaces.

In this case, $$\mathbf{A}^{\max}(T)=\inf C.$$
\end{theorem}

Recall \cite{P} that an operator $S:E\rightarrow F$ is called \textit{$p$-integral} if it belongs to $[\mathcal{N}_{p},\nu_{p}]^{\max}.$ The $p$-integral norm
of $S$ is defined by $i_{p}(S):=\nu_{p}^{\max}(S).$
It follows from Theorem \ref{1.77} that an operator $S:E\rightarrow F$ is $p$-integral if and only if there exists a constant $C>0$ such that
$\nu_{p}(RTS)\leq C\|R\|\|S\|$ for all $S\in \mathcal{F}(G,E)$ and $R\in \mathcal{F}(F,H),$
where $G,H$ are arbitrary Banach spaces. Moreover, $i_{p}(S)=\inf C.$

In an analogous way, O. I. Zhukova \cite{Zhu} introduced the notion of latticially $p$-integral operators by use of latticially $p$-nuclear operators.

\begin{definition}\cite{Zhu}\label{1.7}
An operator $S:E\rightarrow X$ is called \textit{latticially $p$-integral} if there is a number $C$ such that the inequality $\widetilde{\nu}_{p}(BSA)\leq C\|A\|\|B\|$ is valid for arbitrary $F$ and $Y$ and arbitrary operators $A\in \mathcal{F}(F,E), B\in \mathcal{F}(X,Y)_{+}$.

One set
$$\widetilde{i}_{p}(S)=\inf C.$$
\end{definition}
The class of all latticially $p$-integral operators is denoted by $\widetilde{\mathcal{I}}_{p}$. It easily follows from the left positive ideal property of latticially $p$-nuclear operators that latticially $p$-integral operators also have the left positive ideal property.

Naturally, we introduce the notion of positively $p$-integral operators by means of positively $p$-nuclear operators.

\begin{definition}\label{1.31}
We say that an operator $T: X\rightarrow E$ is \textit{positively $p$-integral} if there exists a constant $C>0$ such that
\begin{center}
$\widetilde{\nu}^{p}(RTS)\leq C\|R\|\|S\|$ for all $S\in \mathcal{F}_{+}(Y,X), R\in \mathcal{F}(E,F)$,
\end{center}
where $Y$ is arbitrary Banach lattice and $F$ is arbitrary Banach space.

We put $$\widetilde{i}^{p}(T):=\inf C.$$
\end{definition}
The class of all positively $p$-integral operators is denoted by $\widetilde{\mathcal{I}}^{p}$. It follows from Proposition \ref{1.1} that positively $p$-integral operators have the right positive ideal property.
Clearly, every positively $p$-nuclear operator is positively $p$-integral with $\widetilde{i}^{p}\leq \widetilde{\nu}^{p}$.

The definitions of latticially $p$-integral operators and positively $p$-integral operators both stem from another characterization of the maximal hull of Banach operator ideals (Theorem \ref{1.77}), not from the original definition of the maximal hull (Definition \ref{1.777}). But the following result shows that the class of positively $p$-integral operators coincides with the right positive maximal hull of positively $p$-nuclear operators under the hypothesis of order completeness.

\begin{theorem}\label{3.2}
Let $X$ be an order complete Banach lattice and $E$ be a Banach space. Let $C>0$ and $T\in \mathcal{L}(X,E)$. The following statements are equivalent:
\item[(a)]$T$ is positively $p$-integral with $\widetilde{i}^{p}(T)\leq C$;
\item[(b)]$\widetilde{\nu}^{p}(Q_{L}Ti_{G})\leq C$ for all $G\in LDim(X), L\in COFIN(E)$.
\end{theorem}
\begin{proof}
The implication $(a)\Rightarrow (b)$ is trivial.

\noindent $(b)\Rightarrow (a)$. Given any finite-rank operator $S:E\rightarrow F$ and positive finite-rank operator $R:Y\rightarrow X$. Let $M=RY$. Let $\epsilon>0$.
It follows from Lemma \ref{1.4} that there exist a sublattice $Z$ of $X$ containing $M$, a finite-dimensional sublattice $G$ of $Z$
and a positive projection $P$ from $Z$ onto $G$ such that $\|Px-x\|\leq \epsilon\|x\|$ for all $x\in M$. We define an operator $\widehat{S}:E/\textrm{Ker}(S)\rightarrow F$ by $u+\textrm{Ker}(S)\mapsto Su$. Clearly, the operator $\widehat{S}$ is one-to-one, has the same range as $S$ and $\|\widehat{S}\|=\|S\|$. Then $L:=\textrm{Ker}(S)$ is finite co-dimensional and $S=\widehat{S}Q_{L}$. By (b), we get
\begin{align*}
\widetilde{\nu}^{p}(STPR)&=\widetilde{\nu}^{p}(\widehat{S}Q_{L}Ti_{G}PR)\\
&\leq C\|\widehat{S}\|\|PR\|\\
&\leq (1+\epsilon)C\|S\|\|R\|.
\end{align*}
By Proposition \ref{1.1}, we have
\begin{align*}
\widetilde{\nu}^{p}(STR)&\leq \widetilde{\nu}^{p}(STR-STPR)+\widetilde{\nu}^{p}(STPR)\\
&\leq \widetilde{\nu}^{p}(STR-STPR)+(1+\epsilon)C\|S\|\|R\|\\
&\leq \widetilde{\nu}^{1}(STR-STPR)+(1+\epsilon)C\|S\|\|R\|\\
&\leq 2|\widetilde{\nu}^{1}|(STR-STPR)+(1+\epsilon)C\|S\|\|R\|\\
&=2\nu(STR-STPR)+(1+\epsilon)C\|S\|\|R\|\\
&=2\nu(ST)\|R-PR\|+(1+\epsilon)C\|S\|\|R\|\\
&=2\epsilon\nu(ST)\|R\|+(1+\epsilon)C\|S\|\|R\|\\
\end{align*}
Letting $\epsilon\rightarrow 0$, we get $$\widetilde{\nu}^{p}(STR)\leq C\|S\|\|R\|.$$
This completes the proof.

\end{proof}

\begin{theorem}\label{1.30}
Suppose that $X^{*}$ has the $PMAP$. Let $T\in \mathcal{L}(X,E)$ and let $C>0$. The following statements are equivalent:
\item[(i)]$T$ is positively $p$-integral with $\widetilde{i}^{p}(T)\leq C$.
\item[(ii)]$\sup\{\widetilde{\nu}^{p}(Q_{L}T):L\in COFIN(E)\}\leq C$.
\end{theorem}
\begin{proof}
$(ii)\Rightarrow (i)$ is obvious.

$(i)\Rightarrow (ii)$. Let $L\in COFIN(E)$ and let $\epsilon>0$. We write $Q_{L}T=\sum\limits_{i=1}^{n}x^{*}_{i}\otimes \phi_{i}, x^{*}_{i}\in X^{*},\phi_{i}\in E/L(i=1,2,\cdots,n)$. Choose $\delta>0$ such that $\delta\sum\limits_{i=1}^{n}\|\phi_{i}\|<\epsilon.$ Since $X^{*}$ has the $PMAP$, it follows from Lemma \ref{3.8} that there exists an operator $A\in \mathcal{F}_{+}(X)$ with $\|A\|\leq 1$ such that $\|A^{*}x^{*}_{i}-x^{*}_{i}\|<\delta$ for all $i=1,2,\cdots,n$. By $(i)$, we get
$\widetilde{\nu}^{p}(Q_{L}TA)\leq C.$

By Proposition \ref{1.1}, we get
\begin{align*}
\widetilde{\nu}^{p}(Q_{L}T)&\leq \widetilde{\nu}^{p}(Q_{L}T-Q_{L}TA)+\widetilde{\nu}^{p}(Q_{L}TA)\\
&\leq \widetilde{\nu}^{1}(Q_{L}T-Q_{L}TA)+C\\
&\leq 2|\widetilde{\nu}^{1}|(Q_{L}T-Q_{L}TA)+C\\
&=2\nu(Q_{L}T-Q_{L}TA)+C\\
&\leq 2\sum_{i=1}^{n}\|Ax^{*}_{i}-x^{*}_{i}\|\|\phi_{i}\|+C\\
&\leq 2\epsilon+C.\\
\end{align*}
Letting $\epsilon\rightarrow 0$, we get $\widetilde{\nu}^{p}(Q_{L}T)\leq C$.

\end{proof}

An analogous argument shows the following theorem.

\begin{theorem}\label{2.3}
Suppose that $X$ has the $PMAP$. Let $S\in \mathcal{L}(E,X)$ and let $C>0$. The following statements are equivalent:
\item[(i)]$S$ is latticially $p$-integral with $\widetilde{i}_{p}(S)\leq C$.
\item[(ii)]$\sup\{\widetilde{\nu}_{p}(Si_{M}):M\in FIN(E)\}\leq C$.
\end{theorem}

\begin{corollary}\label{1.20}
\item[(a)]If $S:E\rightarrow X$ is latticially $p$-integral, then $S^{*}$ is positively $p$-integral. In this case, $\widetilde{i}^{p}(S^{*})\leq \widetilde{i}_{p}(S)$.
\item[(b)]If $T:X\rightarrow E$ is positively $p$-integral and $X^{*}$ has the $PMAP$, then $T^{*}$ is latticially $p$-integral.
In this case, $\widetilde{i}_{p}(T^{*})\leq \widetilde{i}^{p}(T).$
\end{corollary}
\begin{proof}
(a). Given any $A\in \mathcal{F}_{+}(Y,X^{*}), B\in \mathcal{F}(E^{*},F)$. We may assume that $F$ is finite-dimensional. Let $\epsilon>0$. By \cite[Lemma 3.1]{JRZ},
there exists a $weak^{*}$-continuous operator $C:E^{*}\rightarrow F$ such that $\|C\|\leq (1+\epsilon)\|B\|$ and $C|_{S^{*}AY}=B|_{S^{*}AY}$. Let $D:F^{*}\rightarrow E$ be an operator such that $D^{*}=C$. Since $S$ is latticially $p$-integral, we get
$$\widetilde{\nu}_{p}(A^{*}J_{X}SD)\leq \|A^{*}J_{X}\|\widetilde{i}_{p}(S)\|D\|\leq (1+\epsilon)\|A\|\widetilde{i}_{p}(S)\|B\|.$$
Clearly, $BS^{*}A=(A^{*}J_{X}SD)^{*}J_{Y}$. By Proposition \ref{1.1} (e), we get
\begin{align*}
\widetilde{\nu}^{p}(BS^{*}A)&\leq \widetilde{\nu}^{p}((A^{*}J_{X}SD)^{*})\\
&\leq \widetilde{\nu}_{p}(A^{*}J_{X}SD)\\
&\leq (1+\epsilon)\|A\|\widetilde{i}_{p}(S)\|B\|.
\end{align*}
Letting $\epsilon\rightarrow 0$, we get $$\widetilde{\nu}^{p}(BS^{*}A)\leq \|A\|\widetilde{i}_{p}(S)\|B\|.$$
Hence, $S^{*}$ is positively $p$-integral and $\widetilde{i}^{p}(S^{*})\leq \widetilde{i}_{p}(S)$.

(b). Given $M\in FIN(E^{*})$ and $\epsilon>0$. We let $L:={}^{\perp}\!M=\{u\in E:\langle u^{*},u\rangle=0$ for all $u^{*}\in M\}$. Then $L\in COFIN(E)$. Note that
$Q_{L}^{*}:(E/L)^{*}\rightarrow E^{*}$ is an isometric embedding and the range of $Q_{L}^{*}$ is $L^{\perp}=M$. Let us define an operator $A:M\rightarrow (E/L)^{*}$ by $Au^{*}=(Q_{L}^{*})^{-1}(u^{*})(u^{*}\in M)$. Clearly, $\|A\|=1$ and $Q^{*}_{L}A=i_{M}$. By Proposition \ref{1.1} $(f)$ and Theorem \ref{1.30}, we get
\begin{align*}
\widetilde{\nu}_{p}(T^{*}i_{M})&=\widetilde{\nu}_{p}(T^{*}Q^{*}_{L}A)\\
&\leq \widetilde{\nu}_{p}(T^{*}Q^{*}_{L})\\
&\leq \widetilde{\nu}^{p}(Q_{L}T)\\
&\leq \widetilde{i}^{p}(T).
\end{align*}
By Theorem \ref{2.3}, $T^{*}$ is latticially $p$-integral and $\widetilde{i}_{p}(T^{*})\leq \widetilde{i}^{p}(T).$

\end{proof}

The following result is immediate from Definition \ref{1.7} and Definition \ref{1.31}.

\begin{lemma}\label{1.8}
\item[(a)]If $S^{**}:E^{**}\rightarrow X^{**}$ is latticially $p$-integral, then so is $S$. In this case,
$\widetilde{i}_{p}(S)\leq \widetilde{i}_{p}(S^{**}).$
\item[(b)]If $T^{**}:X^{**}\rightarrow E^{**}$ is positively $p$-integral, then so is $T$. In this case,
$\widetilde{i}^{p}(T)\leq \widetilde{i}^{p}(T^{**}).$
\end{lemma}

Combining Corollary \ref{1.20} and Lemma \ref{1.8}, we obtain the following two corollaries.

\begin{corollary}\label{1.21}
Suppose that $X^{**}$ has the $PMAP$. The following are equivalent for an operator $S:E\rightarrow X$:
\item[(1)]$S$ is latticially $p$-integral.
\item[(2)]$S^{*}$ is positively $p$-integral.
\item[(3)]$S^{**}$ is latticially $p$-integral.

\noindent In this case, $\widetilde{i}_{p}(S)=\widetilde{i}^{p}(S^{*})=\widetilde{i}_{p}(S^{**}).$
\end{corollary}

\begin{corollary}\label{1.219}
Suppose that $X^{*}$ has the $PMAP$. The following are equivalent for an operator $T:X\rightarrow E$:
\item[(1)]$T$ is positively $p$-integral.
\item[(2)]$T^{*}$ is latticially $p$-integral.
\item[(3)]$T^{**}$ is positively $p$-integral.

\noindent In this case, $\widetilde{i}^{p}(T)=\widetilde{i}_{p}(T^{*})=\widetilde{i}^{p}(T^{**}).$
\end{corollary}

Next we present an important example of positively $p$-integral operators.

\begin{theorem}\label{3.6}
Let $(\Omega, \Sigma, \mu)$ be a probability measure space and $1\leq p<\infty$. Then the inclusion map $i_{p}: L_{p^{*}}(\mu)\rightarrow L_{1}(\mu)$ is positively $p$-integral with $\widetilde{i}^{p}(i_{p})\leq 1$.
\end{theorem}

To prove Theorem \ref{3.6}, we need the following three elementary lemmas.

Let $\tau=(A_{i})_{i=1}^{n}$ be a partition of a probability measure space $(\Omega, \Sigma, \mu)$. We define an operator $$Q_{\tau}: L_{p^{*}}(\mu)\rightarrow L_{1}(\mu), \quad g\mapsto \sum_{i=1}^{n}\frac{\int_{A_{i}}gd\mu}{\mu(A_{i})}\chi_{A_{i}},$$ where $\frac{\int_{A_{i}}gd\mu}{\mu(A_{i})}=0$ if $\mu(A_{i})=0$.
It is easy to see that $\|Q_{\tau}\|=1$.

\begin{lemma}\label{3.1}
$\widetilde{\nu}^{p}(Q_{\tau})=1.$
\end{lemma}
\begin{proof}
Let $f_{i}=\frac{\chi_{A_{i}}}{\mu(A_{i})^{\frac{1}{p^{*}}}}(i=1,2,\cdots,n)$. Then $\|(f_{i})_{i=1}^{n}\|_{p}=1$. For each $i$, we define $\varphi_{i}\in (L_{p^{*}}(\mu))^{*}$ by $\langle \varphi_{i}, g\rangle=\frac{\int_{A_{i}}gd\mu}{\mu(A_{i})^{\frac{1}{p}}}(g\in L_{p^{*}}(\mu)).$ Then $Q_{\tau}=\sum\limits_{i=1}^{n}\varphi_{i}\otimes f_{i}$. Let us define an operator $T: L_{p^{*}}(\mu)\rightarrow l_{p^{*}}$ by $Tg=(\langle \varphi_{i},g\rangle)_{i=1}^{n}$ for $g\in L_{p^{*}}(\mu)$. Note that
$$|\langle\varphi_{i},g\rangle|\leq \frac{\int_{A_{i}}|g|d\mu}{\mu(A_{i})^{\frac{1}{p}}}
\leq \frac{(\int_{\Omega}|g\chi_{A_{i}}|^{p^{*}}d\mu)^{\frac{1}{p^{*}}}(\int_{\Omega}\chi_{A_{i}}d\mu)^{\frac{1}{p}}}{\mu(A_{i})^{\frac{1}{p}}}
=(\int_{A_{i}}|g|^{p^{*}}d\mu)^{\frac{1}{p^{*}}}.$$
Hence $$\sum_{i=1}^{n}|\langle\varphi_{i},g\rangle|^{p^{*}}\leq \sum_{i=1}^{n}\int_{A_{i}}|g|^{p^{*}}d\mu=\int_{\Omega}|g|^{p^{*}}d\mu.$$
This implies $$\|(\varphi_{i})_{i=1}^{n}\|^{w}_{p^{*}}=\|T\|\leq 1.$$
Consequently $$\widetilde{\nu}^{p}(Q_{\tau})\leq \|(\varphi_{i})_{i=1}^{n}\|^{w}_{p^{*}}\cdot \|(f_{i})_{i=1}^{n}\|_{p}\leq 1.$$
Since $\|Q_{\tau}\|=1$, we get $\widetilde{\nu}^{p}(Q_{\tau})=1.$
\end{proof}

The following lemma may be known. For the sake of completeness, we include the proof here.

\begin{lemma}\label{3.5}
Let $f_{1},f_{2},\cdots,f_{n}\in L_{\infty}(\mu)$. Then, for every $\epsilon>0$, there exists a partition $\tau=(A_{i})_{i=1}^{m}$ of $\Omega$ such that
$$\|f_{j}-\sum_{i=1}^{m}\frac{\int_{A_{i}}f_{j}d\mu}{\mu(A_{i})}\chi_{A_{i}}\|_{p}<\epsilon, \quad j=1,2,\cdots,n.$$
\end{lemma}
\begin{proof}
We only prove the conclusion for $n=2$. Other cases are analogous.

We may assume that $f_{1},f_{2}$ are bounded. We set
$$\alpha=\min(\min_{t\in \Omega}f_{1}(t), \min_{t\in \Omega}f_{2}(t))$$
and
$$\beta=\max(\max_{t\in \Omega}f_{1}(t), \max_{t\in \Omega}f_{2}(t)).$$
We choose $a_{0}<a_{1}<\cdots<a_{m}$ such that $$[\alpha,\beta]\subseteq \bigcup\limits_{i=1}^{m}(a_{i-1},a_{i}], \quad a_{i}-a_{i-1}<\epsilon, i=1,2,\cdots,m.$$
Let $A_{ij}=f_{1}^{-1}((a_{i-1},a_{i}])\cap f_{2}^{-1}((a_{j-1},a_{j}])(i,j=1,2,\cdots,m)$. Then $(A_{ij})_{i,j=1}^{m}$ is a partition of $\Omega$.

Note that $$a_{i-1}\mu(A_{ij})\leq \int_{A_{ij}}f_{1}d\mu\leq a_{i}\mu(A_{ij}),\quad i,j=1,2,\cdots,m$$ and hence
$$|f_{1}(t)-\frac{\int_{A_{ij}}f_{1}d\mu}{\mu(A_{ij})}|\leq \epsilon, \quad t\in A_{ij}, i,j=1,2,\cdots,m.$$
This means
$$\int_{\Omega}|f_{1}-\sum_{i,j=1}^{m}\frac{\int_{A_{ij}}f_{1}d\mu}{\mu(A_{ij})}\chi_{A_{ij}}|^{p}d\mu=\sum_{i,j=1}^{m}\int_{A_{ij}}
|f_{1}-\frac{\int_{A_{ij}}f_{1}d\mu}{\mu(A_{ij})}|^{p}d\mu\leq \epsilon^{p}.$$
That is $$\|f_{1}-\sum_{i,j=1}^{m}\frac{\int_{A_{ij}}f_{1}d\mu}{\mu(A_{ij})}\chi_{A_{ij}}\|_{p}\leq \epsilon.$$
Similarly $$\|f_{2}-\sum_{i,j=1}^{m}\frac{\int_{A_{ij}}f_{2}d\mu}{\mu(A_{ij})}\chi_{A_{ij}}\|_{p}\leq \epsilon.$$

\end{proof}

\begin{lemma}\label{3.3}
Let $E$ be a Banach space and let $T\in \mathcal{F}(L_{1}(\mu),E)$. Then, for every $\epsilon>0$, there exists a partition $\tau=(A_{i})_{i=1}^{m}$ of $\Omega$ such that $\nu(Ti_{p}-TQ_{\tau})<\epsilon.$
\end{lemma}
\begin{proof}
We write $T=\sum\limits_{j=1}^{n}f_{j}\otimes u_{j}, f_{j}\in L_{\infty}(\mu), u_{j}\in E, j=1,2,\cdots,n$.
It follows from Lemma \ref{3.5} that there exists a partition $\tau=(A_{i})_{i=1}^{m}$ of $\Omega$ such that
$$\|f_{j}-\sum_{i=1}^{m}\frac{\int_{A_{i}}f_{j}d\mu}{\mu(A_{i})}\chi_{A_{i}}\|_{p}<\frac{\epsilon}{\sum\limits_{i=1}^{n}\|u_{i}\|}, \quad j=1,2,\cdots,n.$$
Hence
\begin{align*}
\nu(Ti_{p}-TQ_{\tau})&\leq \sum\limits_{j=1}^{n}\|f_{j}-Q^{*}_{\tau}f_{j}\|\|u_{j}\|\\
&=\sum\limits_{j=1}^{n}\|f_{j}-\sum_{i=1}^{m}\frac{\int_{A_{i}}f_{j}d\mu}{\mu(A_{i})}\chi_{A_{i}}\|_{p}\|u_{j}\|\\
&<\epsilon.
\end{align*}

\end{proof}

{\em Proof of Theorem \ref{3.6}}.

Given any finite-rank operator $R: L_{1}(\mu)\rightarrow E$ and positive finite-rank operator $S: X\rightarrow L_{p^{*}}(\mu)$. Let $\epsilon>0$. According to Lemma \ref{3.3}, there exists a partition $\tau=(A_{i})_{i=1}^{m}$ of $\Omega$ such that $\nu(Ri_{p}-RQ_{\tau})<\epsilon.$
By Proposition \ref{1.1} and Lemma \ref{3.1}, we get
\begin{align*}
\widetilde{\nu}^{p}(Ri_{p}S)&\leq \widetilde{\nu}^{p}(Ri_{p}S-RQ_{\tau}S)+\widetilde{\nu}^{p}(RQ_{\tau}S)\\
&\leq \widetilde{\nu}^{p}(Ri_{p}-RQ_{\tau})\|S\|+\|R\|\widetilde{\nu}^{p}(Q_{\tau})\|S\|\\
&\leq \widetilde{\nu}^{1}(Ri_{p}-RQ_{\tau})\|S\|+\|R\|\|S\|\\
&\leq 2|\widetilde{\nu}^{1}|(Ri_{p}-RQ_{\tau})\|S\|+\|R\|\|S\|\\
&= 2\nu(Ri_{p}-RQ_{\tau})\|S\|+\|R\|\|S\|\\
&\leq 2\epsilon \|S\|+\|R\|\|S\|\\
\end{align*}

Letting $\epsilon\rightarrow 0$, we get $$\widetilde{\nu}^{p}(Ri_{p}S)\leq \|R\|\|S\|.$$
This completes the proof.  \hfill $\Box$

\begin{remark}
It was known (see \cite[Example 2.9 (b), Corollary 2.8]{DJT} for instance) that the canonical map $j_{p}$ from $C(K)$ to $L_{p}(\mu)$($\mu$- regular Borel measure on compact Hausdorff space $K$) is $p$-integral. O. I. Zhukova \cite{Zhu} strengthened this result and proved that $j_{p}$ is latticially $p$-integral. Although the adjoint of $i_{p}$ is the inclusion map of $L_{\infty}(\mu)$ into $L_{p}(\mu)$, it seems that there is no implication between O. I. Zhukova's result and
Theorem \ref{3.6}, even by Corollaries \ref{1.21} and \ref{1.219}.
\end{remark}

We'll reveal a close relationship between positively $p$-nuclear operators and positively $p$-integral operators. We need two lemmas.

\begin{lemma}\label{3.7}
Suppose that $X^{*}$ has the $PMAP$ and $E$ is a Banach space. Let $T\in \widetilde{\mathcal{N}}^{p}(X,E).$ Then, for every $\epsilon>0$, there exists an operator $R\in \mathcal{F}_{+}(X)$ with $\|R\|\leq 1$ such that $\widetilde{\nu}^{p}(T-TR)<\epsilon$.
\end{lemma}
\begin{proof}
Let $\epsilon>0$. We choose $\delta>0$ with $2\delta+\delta(1+\delta)^{2}\widetilde{\nu}^{p}(T)<\epsilon.$ We choose a positively $p$-nuclear representation
$T=\sum\limits_{j=1}^{\infty}x^{*}_{j}\otimes u_{j}$
such that $$\|(x^{*}_{j})_{j=1}^{\infty}\|_{p^{*}}^{w}\|(u_{j})_{j=1}^{\infty}\|_{p}\leq (1+\delta)\widetilde{\nu}^{p}(T).$$
We may assume that $\|(x^{*}_{j})_{j}\|_{p^{*}}^{w}=1$. We choose $1\leq \xi_{j}\rightarrow \infty$ such that
$$\|(\xi_{j}u_{j})_{j=1}^{\infty}\|_{p}\leq (1+\delta)\|(u_{j})_{j=1}^{\infty}\|_{p}.$$
Choose a positive integral $N$ with $(\sum\limits_{j=N+1}^{\infty}\|u_{j}\|^{p})^{\frac{1}{p}}<\delta$ and also choose a positive real $\eta>0$ with $\eta N^{\frac{1}{p^{*}}}<\delta$. Since $X^{*}$ has the $PMAP$, it follows from Lemma \ref{3.8} that there exists an operator $R\in \mathcal{F}_{+}(X)$ with $\|R\|\leq 1$ such that $\|R^{*}(\frac{x^{*}_{j}}{\xi_{j}})-\frac{x^{*}_{j}}{\xi_{j}}\|<\eta$ for all $j$.
Note that $$T-TR=\sum_{j=1}^{N}(x^{*}_{j}-R^{*}x^{*}_{j})\otimes u_{j}+\sum_{j=N+1}^{\infty}(x^{*}_{j}-R^{*}x^{*}_{j})\otimes u_{j}.$$
Hence
\begin{align*}
\widetilde{\nu}^{p}(T-TR)&\leq \widetilde{\nu}^{p}(\sum_{j=1}^{N}(\frac{x^{*}_{j}}{\xi_{j}}-R^{*}(\frac{x^{*}_{j}}{\xi_{j}}))\otimes \xi_{j}u_{j})+
\widetilde{\nu}^{p}(\sum_{j=N+1}^{\infty}(x^{*}_{j}-R^{*}x^{*}_{j})\otimes u_{j})\\
&\leq \|(\frac{x^{*}_{j}}{\xi_{j}}-R^{*}(\frac{x^{*}_{j}}{\xi_{j}}))_{j=1}^{N}\|_{p^{*}}^{w}\|(\xi_{j}u_{j})_{j=1}^{N}\|_{p}+
\|(x^{*}_{j}-R^{*}x^{*}_{j})_{j=N+1}^{\infty}\|_{p^{*}}^{w}\|(u_{j})_{j=N+1}^{\infty}\|_{p}\\
&\leq (\sum_{j=1}^{N}\|\frac{x^{*}_{j}}{\xi_{j}}-R^{*}(\frac{x^{*}_{j}}{\xi_{j}})\|^{p^{*}})^{\frac{1}{p^{*}}}(1+\delta)\|(u_{j})_{j=1}^{\infty}\|_{p}+
2\|(x^{*}_{j})_{j=1}^{\infty}\|_{p^{*}}^{w}\|(u_{j})_{j=N+1}^{\infty}\|_{p}\\
&\leq \eta N^{\frac{1}{p^{*}}}(1+\delta)^{2}\widetilde{\nu}^{p}(T)+2\delta\\
&\leq \delta (1+\delta)^{2}\widetilde{\nu}^{p}(T)+2\delta\\
&<\epsilon,
\end{align*}
which completes the proof.

\end{proof}

\begin{lemma}\label{3.9}
Suppose that $E$ has the $MAP$ and $X$ is a Banach lattice. Let $T\in \widetilde{\mathcal{N}}^{p}(X,E).$ Then, for every $\epsilon>0$, there exists an operator $S\in \mathcal{F}(E)$ with $\|S\|\leq 1$ such that $\widetilde{\nu}^{p}(T-ST)<\epsilon$.
\end{lemma}
\begin{proof}
Let $\epsilon>0$. Let $\delta>0$ be such that $\delta(1+\delta)^{2}\widetilde{\nu}^{p}(T)<\epsilon.$ We choose a positively $p$-nuclear representation $T=\sum\limits_{j=1}^{\infty}x^{*}_{j}\otimes u_{j}$ such that $$\|(x^{*}_{j})_{j=1}^{\infty}\|_{p^{*}}^{w}\|(u_{j})_{j=1}^{\infty}\|_{p}\leq (1+\delta)\widetilde{\nu}^{p}(T).$$
Choose $1\leq \xi_{j}\rightarrow \infty$ such that
$$\|(\xi_{j}u_{j})_{j=1}^{\infty}\|_{p}\leq (1+\delta)\|(u_{j})_{j=1}^{\infty}\|_{p}.$$
Since $E$ has the $MAP$, there exists an operator $S\in \mathcal{F}(E)$ with $\|S\|\leq 1$ such that
$$\|S(\frac{u_{j}}{\xi_{j}\|u_{j}\|})-\frac{u_{j}}{\xi_{j}\|u_{j}\|}\|<\delta, \quad j=1,2,\cdots.$$
Hence
\begin{align*}
\widetilde{\nu}^{p}(T-ST)&\leq \|(x^{*}_{j})_{j=1}^{\infty}\|_{p^{*}}^{w}\|(u_{j}-Su_{j})_{j=1}^{\infty}\|_{p}\\
&\leq \|(x^{*}_{j})_{j=1}^{\infty}\|_{p^{*}}^{w}\delta(1+\delta)\|(u_{j})_{j=1}^{\infty}\|_{p}\\
&\leq \delta(1+\delta)^{2}\widetilde{\nu}^{p}(T)\\
&<\epsilon.
\end{align*}
This finishes the proof.

\end{proof}

\begin{theorem}\label{3.4}
Suppose that $X^{*}$ has the $PMAP$ and $E$ has the $MAP$. Then
\begin{center}
$\widetilde{\nu}^{p}(T)=\widetilde{i}^{p}(T)$ for all $T\in \widetilde{\mathcal{N}}^{p}(X,E).$
\end{center}
\end{theorem}
\begin{proof}
Let $T\in \widetilde{\mathcal{N}}^{p}(X,E).$ It suffices to show that $\widetilde{\nu}^{p}(T)\leq\widetilde{i}^{p}(T)$.

Let $\epsilon>0$. By Lemma \ref{3.7}, there exists an operator $R\in \mathcal{F}_{+}(X)$ with $\|R\|\leq 1$ such that $\widetilde{\nu}^{p}(T-TR)<\epsilon$. Applying Lemma \ref{3.9} to $TR$, there exists an operator $S\in \mathcal{F}(E)$ with $\|S\|\leq 1$ such that $\widetilde{\nu}^{p}(TR-STR)<\epsilon$. Thus, we get
\begin{align*}
\widetilde{\nu}^{p}(T)&\leq \widetilde{\nu}^{p}(T-TR)+\widetilde{\nu}^{p}(TR-STR)+\widetilde{\nu}^{p}(STR)\\
&\leq 2\epsilon+\widetilde{\nu}^{p}(STR)\\
&\leq 2\epsilon+\|S\|\widetilde{i}^{p}(T)\|R\|\\
&\leq 2\epsilon+\widetilde{i}^{p}(T).\\
\end{align*}
Letting $\epsilon\rightarrow 0$, we get $$\widetilde{\nu}^{p}(T)\leq\widetilde{i}^{p}(T),$$
which completes the proof.

\end{proof}

To describe the space of positively $p$-integral operators, we set $$\Upsilon_{p}^{0}(E,X):=\overline{\mathcal{F}(E,X)}^{\|\cdot\|_{\Upsilon_{p}}}.$$

\begin{lemma}\label{2.6}
Suppose that $E^{**}$ has the $MAP$ and $X^{*}$ has the $PMAP$. Let $S\in \Upsilon_{p}^{0}(E,X)$ and let $T\in \widetilde{\mathcal{I}}^{p^{*}}(X,E^{**})$. Then $TS$ is nuclear and $\nu(TS)\leq \widetilde{i}^{p^{*}}(T)\|S\|_{\Upsilon_{p}}.$
\end{lemma}
\begin{proof}
Case 1. $S$ is finite-rank.

Since $E^{**}$ has the $MAP$, $E^{*}$ also has the $MAP$.
By \cite[Proposition 10.3.1]{P}, we get $$\nu(TS)=\sup\{|\textrm{trace}(RTS)|:R\in \mathcal{L}(E^{**}),\|R\|\leq 1\}.$$

Since $E^{**}$ has the $MAP$, we get $$\sup\{|\textrm{trace}(RTS)|:R\in \mathcal{L}(E^{**}),\|R\|\leq 1\}=\sup\{|\textrm{trace}(RTS)|:R\in \mathcal{F}(E^{**}),\|R\|\leq 1\}.$$

Theorem \ref{2.1} and Theorem \ref{3.4} yield
\begin{align*}
\sup\{|\textrm{trace}(RTS)|:R\in \mathcal{F}(E^{**}),\|R\|\leq 1\}&\leq \sup\{\widetilde{\nu}^{p^{*}}(RT)\|S\|_{\Upsilon_{p}}:R\in \mathcal{F}(E^{**}),\|R\|\leq 1\}\\
&=\sup\{\widetilde{i}^{p^{*}}(RT)\|S\|_{\Upsilon_{p}}:R\in \mathcal{F}(E^{**}),\|R\|\leq 1\}\\
&\leq\sup\{\|R\|\widetilde{i}^{p^{*}}(T)\|S\|_{\Upsilon_{p}}:R\in \mathcal{F}(E^{**}),\|R\|\leq 1\}\\
&\leq \widetilde{i}^{p^{*}}(T)\|S\|_{\Upsilon_{p}}.
\end{align*}
Hence, we get $$\nu(TS)\leq \widetilde{i}^{p^{*}}(T)\|S\|_{\Upsilon_{p}}.$$

Case 2. $S\in \Upsilon_{p}^{0}(E,X)$.

Let $\epsilon>0$. Then there exists a sequence $(S_{n})_{n}$ in $\mathcal{F}(E,X)$ such that
\begin{center}
$\sum\limits_{n}S_{n}=S$ in $\|\cdot\|_{\Upsilon_{p}}$ and $\sum\limits_{n}\|S_{n}\|_{\Upsilon_{p}}\leq (1+\epsilon)\|S\|_{\Upsilon_{p}}.$
\end{center}
By Case 1,
\begin{center}
$\nu(TS_{n})\leq \widetilde{i}^{p^{*}}(T)\|S_{n}\|_{\Upsilon_{p}}$ for all $n$.
\end{center}
This implies
$$\sum_{n}\nu(TS_{n})\leq (1+\epsilon)\widetilde{i}^{p^{*}}(T)\|S\|_{\Upsilon_{p}}.$$
Hence
\begin{center}
$\sum\limits_{n}TS_{n}=U$ in $\nu$ for some $U\in \mathcal{N}(E,E^{**})$.
\end{center}
and so
\begin{center}
$\sum\limits_{n}TS_{n}=U$ in operator norm $\|\cdot\|$.
\end{center}
Note that
\begin{center}
$\sum\limits_{n}S_{n}=S$ in operator norm $\|\cdot\|$.
\end{center}
Therefore, we get $TS=U\in\mathcal{N}(E,E^{**})$. Moreover, $$\nu(TS)=\nu(U)\leq(1+\epsilon)\widetilde{i}^{p^{*}}(T)\|S\|_{\Upsilon_{p}}.$$
Letting $\epsilon\rightarrow 0$, we get $$\nu(TS)\leq\widetilde{i}^{p^{*}}(T)\|S\|_{\Upsilon_{p}}.$$
\end{proof}

\begin{lemma}\label{2.8}
Let $T\in \mathcal{F}(X,E)$.
\item[(a)]If $E$ has the $MAP$ or $X^{*}$ has the $PMAP$, then
$$\widetilde{\nu}^{p}(T)=\sup\{|\textrm{trace}(RT)|:R\in \mathcal{F}(E,X),\|R\|_{\Upsilon_{p^{*}}}\leq 1\}.$$
\item[(b)]If $E=F^{**}$ has the $MAP$, then
$$\widetilde{\nu}^{p}(T)=\sup\{|\textrm{trace}(TS)|:S\in \mathcal{F}(F,X),\|S\|_{\Upsilon_{p^{*}}}\leq 1\}.$$
\end{lemma}
\begin{proof}
(a). By Theorem \ref{1.13}, we get $$\widetilde{\nu}^{p}(T)=\sup\{|\textrm{trace}(ST)|:S\in \Upsilon_{p^{*}}(E,X^{**}),\|S\|_{\Upsilon_{p^{*}}}\leq 1\}.$$
We set $$c_{T}:=\sup\{|\textrm{trace}(RT)|:R\in \mathcal{F}(E,X),\|R\|_{\Upsilon_{p^{*}}}\leq 1\}.$$
Clearly, $c_{T}\leq \widetilde{\nu}^{p}(T).$
It remains to prove the reverse. Let $S\in \Upsilon_{p^{*}}(E,X^{**}),\|S\|_{\Upsilon_{p^{*}}}\leq 1$.

Case 1. $E$ has the $MAP$.

Let $\epsilon>0$. Then there exists an operator $A\in \mathcal{F}(E),\|A\|\leq 1+\epsilon$ such that $AT=T$. We let $R=SA$ and write
$$R=\sum_{i=1}^{n}u^{*}_{i}\otimes x^{**}_{i}, \quad u^{*}_{i}\in E^{*},x^{**}_{i}\in X^{**},i=1,2,\cdots,n.$$
and $$T=\sum_{j=1}^{m}x^{*}_{j}\otimes u_{j}, \quad x^{*}_{j}\in X^{*}, u_{j}\in E, j=1,2,\cdots,m.$$
We choose $\delta>0$ such that
\begin{center}
$\delta(2+\delta)\sum_{j=1}^{m}\sum_{i=1}^{n}\|u^{*}_{i}\|\|u_{j}\|\|x^{**}_{i}\|\|x^{*}_{j}\|<\epsilon$ and $(1+\delta)^{2}\leq 1+\epsilon.$
\end{center}
We set $M=\textrm{span}\{x^{**}_{i}:1\leq i\leq n\}$ and $L=\textrm{span}\{x^{*}_{j}:1\leq i\leq m\}.$ It follows from Lemma \ref{1.4} that there exist a sublattice $Z$ of $X^{**}$ containing $M$, a finite-dimensional sublattice $G$ of $Z$ and a positive projection $P$ from $Z$ onto $G$ such that $\|Px^{**}-x^{**}\|\leq \delta\|x^{**}\|$ for all $x^{**}\in M$. By Theorem \ref{1.3}, there exists a lattice isomorphism $B$ from $G$ into $X$ such that $\|B\|,\|B^{-1}\|\leq 1+\delta$ and
$$|\langle x^{**},x^{*}\rangle-\langle x^{*},Bx^{**}\rangle|\leq \delta \|x^{**}\|\|x^{*}\|, \quad x^{**}\in G, x^{*}\in L.$$
Let $\widetilde{R}=BPR\in \mathcal{F}(E,X)$ and then
\begin{align*}
\|\widetilde{R}\|_{\Upsilon_{p^{*}}}&=\|BPSA\|_{\Upsilon_{p^{*}}}\\
&\leq \|BP\|\|S\|_{\Upsilon_{p^{*}}}\|A\|\\
&\leq \|B\|\|P\|\|A\|\\
&\leq (1+\epsilon)^{2}\\
\end{align*}
Note that for all $i,j$, we have
\begin{align*}
|\langle x^{**}_{i},x^{*}_{j}\rangle-\langle x^{*}_{j},BPx^{**}_{i}\rangle|&\leq |\langle x^{**}_{i},x^{*}_{j}\rangle-\langle Px^{**}_{i},x^{*}_{j}\rangle|+
|\langle Px^{**}_{i},x^{*}_{j}\rangle-\langle x^{*}_{j},BPx^{**}_{i}\rangle|\\
&\leq \delta \|x^{**}_{i}\|\|x^{*}_{j}\|+\delta \|Px^{**}_{i}\|\|x^{*}_{j}\|\\
&\leq \delta \|x^{**}_{i}\|\|x^{*}_{j}\|+\delta(1+\delta)\|x^{**}_{i}\|\|x^{*}_{j}\|\\
&=\delta(2+\delta)\|\|x^{**}_{i}\|\|x^{*}_{j}\|\\
\end{align*}
This implies
\begin{align*}
|\textrm{trace}(ST)-\textrm{trace}(\widetilde{R}T)|&=|\textrm{trace}(RT)-\textrm{trace}(\widetilde{R}T)|\\
&=|\sum_{j=1}^{m}\sum_{i=1}^{n}\langle u^{*}_{i},u_{j}\rangle(\langle x^{**}_{i},x^{*}_{j}\rangle-\langle x^{*}_{j},BPx^{**}_{i}\rangle)|\\
&\leq \delta(2+\delta)\sum_{j=1}^{m}\sum_{i=1}^{n}\|u^{*}_{i}\|\|u_{j}\|\|x^{**}_{i}\|\|x^{*}_{j}\|\\
&<\epsilon.
\end{align*}
This yields
$$|\textrm{trace}(ST)|\leq \epsilon+|\textrm{trace}(\widetilde{R}T)|
\leq \epsilon+(1+\epsilon)^{2}c_{T}.$$

Hence
$$\widetilde{\nu}^{p}(T)\leq \epsilon+(1+\epsilon)^{2}c_{T}.$$
Letting $\epsilon\rightarrow 0$, we get
$\widetilde{\nu}^{p}(T)\leq c_{T}.$

Case 2. $X^{*}$ has the $PMAP$.

Let $\epsilon>0$. We write $T=\sum\limits_{i=1}^{n}x^{*}_{i}\otimes u_{i}(x^{*}_{i}\in X^{*},u_{i}\in E, i=1,2,\cdots,n).$ Choose $\delta>0$ with $\delta\sum\limits_{i=1}^{n}\|Su_{i}\|<\epsilon.$ Since $X^{*}$ has the $PAMP$, it follows from Lemma \ref{3.8} that there exists an operator $B\in \mathcal{F}_{+}(X), \|B\|\leq 1$ such that $\|B^{*}x^{*}_{i}-x^{*}_{i}\|<\delta$ for all $i=1,2,\cdots,n$. We set $R=B^{**}S\in \mathcal{F}(E,X)$. Then $
\|R\|_{\Upsilon_{p^{*}}}\leq 1$ and
\begin{align*}
|\textrm{trace}(ST)-\textrm{trace}(RT)|&=|\sum_{i=1}^{n}\langle Su_{i},x^{*}_{i}\rangle-\sum_{i=1}^{n}\langle B^{*}x^{*}_{i},Su_{i}\rangle|\\
&\leq \sum_{i=1}^{n}\|Su_{i}\|\|B^{*}x^{*}_{i}-x^{*}_{i}\|\\
&\leq \delta\sum\limits_{i=1}^{n}\|Su_{i}\|<\epsilon.
\end{align*}
Hence
$$\widetilde{\nu}^{p}(T)\leq \epsilon+c_{T}.$$
Letting $\epsilon\rightarrow 0$, we get
$\widetilde{\nu}^{p}(T)\leq c_{T}.$
This completes the proof of (a).

(b). By (a), it suffices to prove
$$\sup\{|\textrm{trace}(RT)|:R\in \mathcal{F}(E,X),\|R\|_{\Upsilon_{p^{*}}}\leq 1\}=\sup\{|\textrm{trace}(TS)|:S\in \mathcal{F}(F,X),\|S\|_{\Upsilon_{p^{*}}}\leq 1\}.$$
For the sake of convenience, we set $$\alpha:=\sup\{|\textrm{trace}(RT)|:R\in \mathcal{F}(E,X),\|R\|_{\Upsilon_{p^{*}}}\leq 1\}$$ and
$$\beta:=\sup\{|\textrm{trace}(TS)|:S\in \mathcal{F}(F,X),\|S\|_{\Upsilon_{p^{*}}}\leq 1\}.$$
Let $S\in \mathcal{F}(F,X)$ with $\|S\|_{\Upsilon_{p^{*}}}\leq 1$. By Theorem \ref{1.6}, $\|S^{**}\|_{\Upsilon_{p^{*}}}=\|S\|_{\Upsilon_{p^{*}}}\leq 1.$ It is easy to check that $\textrm{trace}(TS)=\textrm{trace}(S^{**}T)$. Hence, we get $\beta\leq \alpha.$

Conversely, let $R\in \mathcal{F}(E,X)$ with $\|R\|_{\Upsilon_{p^{*}}}\leq 1$. Let $\epsilon>0$. We write $T=\sum\limits_{i=1}^{n}x^{*}_{i}\otimes u_{i}$, where $x^{*}_{i}\in X^{*},u_{i}\in E(i=1,2,\cdots,n)$. Choose $\delta>0$ such that $\delta \|R\|\sum\limits_{i=1}^{n}\|x^{*}_{i}\|<\epsilon.$ Since $E$ has the $MAP$, there exists an operator $A\in \mathcal{F}(E)$ with $\|A\|\leq 1$ such that $\|Au_{i}-u_{i}\|<\delta$ for all $i=1,2,\cdots,n.$ Hence we get
\begin{align}\label{16}
|\textrm{trace}(RT)-\textrm{trace}(RAT)|&=|\sum\limits_{i=1}^{n}\langle x^{*}_{i},Ru_{i}\rangle-\sum\limits_{i=1}^{n}\langle x^{*}_{i},RAu_{i}\rangle| \nonumber \\
&\leq \sum\limits_{i=1}^{n}\|x^{*}_{i}\|\|R\|\|Au_{i}-u_{i}\| \nonumber \\
&<\epsilon.
\end{align}
We also write $A=\sum\limits_{j=1}^{m}u^{*}_{j}\otimes w_{j}, u^{*}_{j}\in E^{*}, w_{j}\in E(j=1,2,\cdots,m).$ We set $M=\textrm{span}\{u^{*}_{j}:1\leq j\leq m\}$ and $L=\textrm{span}\{u_{i}:1\leq i\leq n\}$. It follows from the principle of local reflexivity in Banach spaces that there exists an operator $C:M\rightarrow F^{*}$ such that

(i) $C|_{M\cap F^{*}}=I_{M\cap F^{*}}$;

(ii) $(1-\epsilon)\|u^{*}\|\leq \|Cu^{*}\|\leq (1+\epsilon)\|u^{*}\|, \quad u^{*}\in M$;

(iii) $\langle u^{*},u\rangle=\langle u, Cu^{*}\rangle, \quad u^{*}\in M, u\in L.$

We set $B=\sum\limits_{j=1}^{m}Cu^{*}_{j}\otimes w_{j}$ and $S=RB$. Clearly, $CA^{*}=B^{*}$. By (ii), we get
$$\|S\|_{\Upsilon_{p^{*}}}\leq \|R\|_{\Upsilon_{p^{*}}}\|B\|\leq \|B^{*}\|\leq 1+\epsilon.$$
By (ii), it be can verified that $\textrm{trace}(RAT)=\textrm{trace}(TS)$. Thus (\ref{16}) yields that
$$|\textrm{trace}(RT)-\textrm{trace}(TS)|<\epsilon.$$
This implies
$$|\textrm{trace}(RT)|\leq \epsilon+(1+\epsilon)\beta.$$
By the arbitrariness of $R$, we get
$$\alpha\leq \epsilon+(1+\epsilon)\beta.$$
Letting $\epsilon\rightarrow 0$, we get $\alpha\leq \beta.$ This means $\alpha=\beta.$

\end{proof}

\begin{theorem}
Suppose that $E^{**}$ has the $MAP$, $X^{*}$ has the $PMAP$ and $X$ is order continuous. Then
$$\widetilde{\mathcal{I}}^{p^{*}}(X,E^{**})=(\Upsilon_{p}^{0}(E,X))^{*}.$$
\end{theorem}
\begin{proof}
We define an operator $$U: \widetilde{\mathcal{I}}^{p^{*}}(X,E^{**})\rightarrow (\Upsilon_{p}^{0}(E,X))^{*}$$ by
$$T\mapsto U_{T}(S)=\textrm{trace}(TS), \quad T\in \widetilde{\mathcal{I}}^{p^{*}}(X,E^{**}), S\in \Upsilon_{p}^{0}(E,X).$$
By Lemma \ref{2.6}, we get $\|U_{T}\|\leq\widetilde{i}^{p^{*}}(T).$

Let $\varphi\in (\Upsilon_{p}^{0}(E,X))^{*}$. We define an operator $T:X\rightarrow E^{**}$ by $\langle Tx, u^{*}\rangle=\langle\varphi, u^{*}\otimes x\rangle$ for $x\in X, u^{*}\in E^{*}$. Obviously, $\|T\|\leq \|\varphi\|$ and $\langle \varphi,S\rangle=\textrm{trace}(TS)$ for all $S\in \mathcal{F}(E,X)$.

Claim. $T$ is positively $p^{*}$-integral.

Let $G\in LDim(X)$ and $L\in COFIN(E^{**})$. Let $\epsilon>0$. Since $X^{*}$ has the $PMAP$, $X$ also has the $PMAP$. By \cite[Theorem 2.7]{N}, there exists an operator $D\in \mathcal{F}_{+}(X)$ with $\|D\|\leq 1+\epsilon$ such that $D|_{G}=I_{G}$. By Lemma \ref{2.8}(b), we get
\begin{align*}
\widetilde{\nu}^{p^{*}}(Q_{L}Ti_{G})&=\widetilde{\nu}^{p^{*}}(Q_{L}TDi_{G})\\
&\leq\widetilde{\nu}^{p^{*}}(TD)\\
&=\sup\{|\textrm{trace}(TDV)|:V\in \mathcal{F}(E,X),\|V\|_{\Upsilon_{p^{*}}}\leq 1\}\\
&=\sup\{|\langle \varphi,DV\rangle|:V\in \mathcal{F}(E,X),\|V\|_{\Upsilon_{p^{*}}}\leq 1\}\\
&\leq \|\varphi\|\|D\|\\
&\leq (1+\epsilon)\|\varphi\|.
\end{align*}
Letting $\epsilon\rightarrow 0$, we get $\widetilde{\nu}^{p^{*}}(Q_{L}Ti_{G})\leq \|\varphi\|.$
It follows from Theorem \ref{3.2} that $T$ is positively $p^{*}$-integral and $\widetilde{i}^{p^{*}}(T)\leq \|\varphi\|.$

Finally, by the definition of $\Upsilon_{p}^{0}(E,X)$, we see that $\varphi=U_{T}$. Therefore the mapping $U$ is a surjective linear isometry.

\end{proof}

\end{document}